\documentclass[reqno, 11pt, letterpaper ]{amsart}

\usepackage{amsmath}
\usepackage{amsfonts}
\usepackage{amssymb}
\usepackage{graphicx}
\usepackage{amsthm,color,yfonts,cite} \usepackage{paralist}
\usepackage{hyperref}
\usepackage{physics}
\usepackage{bbm}

\newtheorem{theorem}{Theorem}
\newtheorem{definition}[theorem]{Definition}
\newtheorem{proposition}[theorem]{Proposition}
\newtheorem{lemma}[theorem]{Lemma}

\newtheorem{assumption}[theorem]{Assumption}

\theoremstyle{remark}
\newtheorem{remark}[theorem]{Remark}

\makeatletter
\newcommand*{\rom}[1]{\expandafter\@slowromancap\romannumeral #1@}
\makeatother

\newcommand{\R}{\mathbb{R}}

\newcommand{\bchi}{\boldsymbol{\chi}}

\usepackage{wrapfig}
\usepackage{tikz}
\usetikzlibrary{arrows,calc,decorations.pathreplacing}
\definecolor{light-gray1}{gray}{0.90}
\definecolor{light-gray2}{gray}{0.80}
\definecolor{light-gray3}{gray}{0.60}

\numberwithin{equation}{section}

\numberwithin{theorem}{section}

\numberwithin{table}{section}

\numberwithin{figure}{section}

\ifx\pdfoutput\undefined
  \DeclareGraphicsExtensions{.pstex, .eps}
\else
  \ifx\pdfoutput\relax
    \DeclareGraphicsExtensions{.pstex, .eps}
  \else
    \ifnum\pdfoutput>0
      \DeclareGraphicsExtensions{.pdf}
    \else
      \DeclareGraphicsExtensions{.pstex, .eps}
    \fi
  \fi
\fi

\title{On    steady states for the Vlasov-Schr\"odinger-Poisson system}

\date{\today}
\author[Y. Hong]{Younghun Hong}
\address{Department of Mathematics, Chung-Ang University, Seoul 06974, Korea}
\email{yhhong@cau.ac.kr}

\author[S. Jin]{Sangdon Jin}
\address{Department of Mathematics, Chung-Ang University, Seoul 06974, Korea}
\email{sdjin@cau.ac.kr}

\begin{document}
\maketitle
 
\begin{abstract}
The Vlasov-Schr\"odinger-Poisson system is a kinetic-quantum hybrid model describing quasi-lower dimensional electron gases. For this system, we construct a large class of 2D kinetic/1D quantum steady states in a bounded domain as generalized free energy minimizers, and we show their finite subband structure, monotonicity, uniqueness and \textit{conditional} dynamical stability.  Our proof is based on the concentration-compactness principle, but some additional difficulties arise due to lack of compactness originated from the hybrid nature (see Remark \ref{difficulties}). To overcome the difficulties, we introduce a 3-step refinement of a minimizing sequence by rearrangement and partial minimization problems, and the coercivity lemma for the free energy (Lemma \ref{lemma: energy coercivity}) is crucially employed.
\end{abstract}

\section{Introduction}

\subsection{Model description}
The dynamics of a large number of electrons is described by various approximation models depending on the physical scale under consideration. At the microscopic level (atoms, molecules, nanostructures), the Hartree-Fock system provides fairly accurate predictions \cite{EESY04, BPS, BJPSS, PRSS}. On the other hand, at the larger mesoscopic level (gases), microscopic models are reduced to kinetic ones such as Vlasov, Boltzmann, Fokker-Planck type equations, and they are employed for computational efficiency. In particular, the Vlasov-Poisson system is rigorously derived as a semi-classical limit from the Hartree-Fock equation \cite{AKN, APPP, LionsPaul, MarkMau, BPSS}.

In some physical settings, however, gaseous electrons are partially strongly confined in a certain direction \cite{Davies, FG}. In a nanowire, electrons are confined in the diameter direction of order tens of nanometers. Two-dimensional electron gases (2DEGs) are formed at the heterojunction between two semiconductors. In such nanostructures, the kinetic transport and the quantum effects must be taken in account at the same time, because quantized energy levels appear in the confined direction while electrons are transported in the unconfined direction. 

In this article, we are concerned with electron gases in a very thin bounded domain, i.e., 2DEGs, whose thickness is comparable to the electron de Broglie length. 2DEGs are constructed in a modulation-doped field-effect transistor (MODFET), a high-electron-mobility-transistor (HEMT) or a graphene where electrons have extremely high mobility \cite[Chapter 9]{Davies}. In this situation, quantum observables in the Hartree-Fock formalism are reduced to a partly quantum and partly kinetic states via the partial semi-classical limit (see \eqref{partially quantum-mechanical state} below). As a consequence, a quantum-kinetic hybrid model, namely a subband model, is derived (see \cite{BM_semiclassical}). 

To be precise, in a dimensionless form, we take a 3D domain\footnote{The thin domain $\Omega_\epsilon=\omega\times(0,\epsilon)$ is scaled to $\Omega=\omega\times(0,1)$. Accordingly, the Schr\"odinger operator $-\frac{\epsilon^2}{2}\Delta_x$ is transformed to $-\frac{\epsilon^2}{2}\Delta_y-\frac{1}{2}\partial_z^2$. Then, $\frac{|v|^2}{2}-\frac{1}{2}\partial_z^2$ is derived by the formal semi-classical limit.} of the form 
$$\Omega=\omega\times(0,1),$$
where $\omega$ is a smooth bounded domain in $\mathbb{R}^2$. Throughout this article, we use the notation 
$$x=(y,z)\in \omega\times(0,1),$$
and the $z$-variable (resp., $y$-variable) is regarded as the confined (resp., non-confined) directional variable.  We denote a sequence of kinetic distributions on the unconfined phase space by 
$$\mathbf{f}=\{f_j\}_{j=1}^\infty\quad\textup{with }f_j=f_j(y,v):\omega\times\mathbb{R}^2\to[0,\infty),$$
and a sequence of quantum states by 
$$\bchi=\{\chi_j\}_{j=1}^\infty\quad\textup{with }\chi_j=\chi_j(x):\Omega\to\mathbb{C}.$$
We assume that $\bchi$ is \textit{partially orthonormal} (with respect to the confined variable $z$) in the sense that for each $y\in\omega$,
$$\langle \chi_j(y,\cdot),\chi_k(y,\cdot)\rangle_{L^2(0,1)}=\int_0^1\chi_j(y,z)\overline{\chi_k(y,z)}dz=\delta_{jk}\quad\textup{for all }j,k\in\mathbb{N}.$$
Then, partially strongly confined electrons can be described by an operator-valued function 
\begin{equation}\label{partially quantum-mechanical state}
(y,v)\mapsto\sum_{j=1}^\infty f_j(y,v)\ketbra{\chi_j(y,\cdot)}{\chi_j(y,\cdot)},
\end{equation}
where $\ketbra{\chi_j(y,\cdot)}{\chi_j(y,\cdot)}$ is the one-particle projector to $L^2(0,1)$, i.e., 
$$\ketbra{\chi_j(x)}{\chi_j(x)}\varphi(z)\rangle=\chi_j(x)\int_\omega \overline{\chi_j(y,z)}\varphi(z)dz\quad\textup{for all }\varphi\in L^2(0,1).$$
Note that a hybrid state \eqref{partially quantum-mechanical state} is equivalent to a pair $(\mathbf{f},\bchi)$, where $(f_j, \chi_j)$ represents the \textit{$j$-th subband}.

In this context, the mass and the energy for the electronic Hartree-Fock system lead to the partially kinetic mass and the energy functionals\footnote{In the definitions, $\bchi$ is hidden in several places. Indeed, the proper definition of the mass would be $\sum_{j=1}^\infty \iint_{\Omega \times\mathbb{R}^2}f_j(y,v)|\chi_j(x)|^2dxdv$, but integration in the $z$-variable yields the formula \eqref{mass}. Likewise, the unconfined directional kinetic energy $\frac{1}{2}\sum_{j=1}^\infty \iint_{\Omega \times\mathbb{R}^2}|v|^2f_j(y,v)|\chi_j(x)|^2dxdv$ is reduced to $\frac{1}{2}\sum_{j=1}^\infty \iint_{\omega \times\mathbb{R}^2}|v|^2f_j(y,v)dydv$ in the total energy \eqref{energy}.}, defined by  
\begin{equation}\label{mass}
\mathcal{M}(\mathbf{f}):=\sum_{j=1}^\infty \iint_{\omega \times\mathbb{R}^2}f_j(y,v)dydv
\end{equation}
and
\begin{equation}\label{energy}
\begin{aligned}
\mathcal{E} (\mathbf{f}, \bchi):&=\sum_{j=1}^\infty \iint_{\omega \times\mathbb{R}^2}\left(\frac{|v|^2}{2}+\frac{1}{2}\|\partial_z\chi_j(y,\cdot)\|_{L^2(0,1)}^2\right)f_j(y,v)dydv\\
&\quad+\sum_{j=1}^\infty \iint_{\Omega \times\mathbb{R}^2}V_{\textup{ext}}(x)|\chi_j(x)|^2f_j(y,v)dxdv+\frac{1}{2}\int_{\Omega} |\nabla U_{\rho_{(\mathbf{f},\bchi)}}(x)|^2dx,
\end{aligned}
\end{equation}
where 
\begin{equation}\label{total density function}
\rho_{(\mathbf{f},\bchi)}(x):=\sum_{j=1}^\infty \rho_{f_j}(y)|\chi_j(x)|^2=\sum_{j=1}^\infty \left(\int_{\mathbb{R}^2}f_j(y,v)dv\right)|\chi_j(x)|^2
\end{equation}
is the \textit{total density function} and $U_{\rho_{(\mathbf{f},\bchi)}}$ is the interaction potential solving the Poisson equation (see \eqref{Poisson equation} below) with $\rho=\rho_{(\mathbf{f},\bchi)}$. Then, the associated Hamiltonian dynamics is described by the so-called Vlasov-Schr\"odinger-Poisson system \eqref{Vlasov equation}-\eqref{Poisson equation} below; On the time interval $[0,T)$, a time-dependent kinetic distribution sequence $\mathbf{f}(t)=\{f_j(t,y,v)\}_{j=1}^\infty$ obeys the 2D Vlasov equation 
\begin{equation}\label{Vlasov equation}
\left\{\begin{aligned}
\partial_t f_j+v\cdot\nabla_yf_j-\lambda_j(t,y)\cdot\nabla_v f_j&=0&&\textup{on }(0,T)\times\omega\times\mathbb{R}^2,\\
f_j(0,y,v)&=f_{j,0}(y,v)&&\textup{on }\omega\times\mathbb{R}^2,\\
f_j(t,y,v)&=g_j(t,y,v) &&\textup{on }(0,T)\times\Sigma_-,
\end{aligned}\right.
\end{equation}
where the outgoing set is given by 
$$\Sigma_-=\Big\{(y,v)\in\partial\omega\times\mathbb{R}^2: v\cdot n_y<0\Big\}$$
and $n_y$ is the outgoing normal vector at $y\in\partial\omega$. On the other hand, for each $(t,y)\in[0,T)\times\omega$, a partially orthonormal quantum state sequence $\bchi(t,y,\cdot)=\{\chi_j(t,y,\cdot)\}_{j=1}^\infty$ solves the quasistatic 1D Schr\"odinger equation with respect to the $z$-variable,
\begin{equation}\label{Schrodinger equation}
\left\{\begin{aligned}
\left(-\frac{\partial_z^2}{2}+(U+V_\textup{ext})(t,y,\cdot)\right)\chi_j(t,y,\cdot)&=\lambda_j(t,y)\chi_j(t,y,\cdot)&&\textup{on }(0,1),\\
\chi_j(t,y,z)&=0&&\textup{if }z=0,1,
\end{aligned}\right.
\end{equation}
where $\lambda_j(t,y)$ is the $j$-th eigenvalue of the Schr\"odinger operator $-\frac{\partial_z^2}{2}+(U+V_\textup{ext})(t,y,\cdot)$ acting on the form domain $H_0^1(0,1)$. Finally, for each $t\in[0,T)$, the self-consistent electronic potential $U(t,\cdot)$ satisfies the Poisson equation 
\begin{equation}\label{Poisson equation}
\left\{\begin{aligned}
-\Delta U&=\rho&&\textup{in } \Omega,\\
U&=0&&\textup{on } \partial \omega \times (0,1), \\
\partial_zU&=0&&\textup{on } \omega\times \{0,1\}.
\end{aligned}\right.
\end{equation}
with $\rho=\rho_{(\mathbf{f}(t),\bchi(t))}$.

Including the above quantum-kinetic model, various subband models have been studied extensively by Ben Abdallah and M\'ehats, and their collaborators \cite{Ben, BCCV, BDG, BM, Mehats, BMN, BMQ, BMV, BMV2, Pinaud, Vau}. In particular, in the important work \cite{BM}, existence of weak solutions to the system \eqref{Vlasov equation}-\eqref{Poisson equation} is established provided that $(i)$ initial data for  higher subbands $\{f_{j,0}(y,v)\}_{j=2}^\infty$ are sufficiently small, and $(ii)$ the boundary data $\{g_j(t,y,v)\}_{j=1}^\infty$ restricted to the outgoing set is bounded in certain norms. Nevertheless, many important questions such as existence of general solutions and their uniqueness still remain open.

In this article, for a conserved system, we however impose the specular-reflection boundary condition to the Vlasov equation \eqref{Vlasov equation}, i.e.,
\begin{equation}\label{specular-reflection boundary condition}
f_j(t,y,v)=f_j\big(t,y,v-2(v\cdot n_y)n_y\big)\quad\textup{for all }(t,y,v)\in(0,T)\times\Sigma_-.
\end{equation}
Indeed, under this boundary condition, sufficiently regular solutions (if they exist) preserve the mass and the energy. Moreover, a generalized entropy (or a Casimir functional) of the form 
$$\mathcal{C}_\beta(\mathbf{f})=\sum_{j=1}^\infty\iint_{\omega \times\mathbb{R}^2}\beta\big(f_j(y,v)\big)dydv,$$
with $\beta:[0,1]\to[0,\infty)$, and the corresponding  generalized \textit{free energy} (or the energy-Casimir) functional 
\begin{equation}\label{free energy}
\mathcal{F}(\mathbf{f},\bchi)=\mathcal{F}_{T,\beta}(\mathbf{f},\bchi):=\mathcal{E}(\mathbf{f},\bchi)+T\mathcal{C}_{\beta}(\mathbf{f}),
\end{equation}
are formally conserved, because each $f_j(t)$ is volume-preserving. 

\subsection{Statement of the main result}
The goal of this paper is to construct a large class of  stationary states to the Vlasov-Poisson-Schr\"odinger system \eqref{Vlasov equation}-\eqref{Poisson equation} by minimizing generalized free energies under a mass constraint. For a proper formation of the variational problem, we introduce the notion of an admissible pair.

\begin{definition}
\begin{enumerate}[$(i)$]
\item (Kinetic admissible class $\mathcal{A}_{\textup{c.m.}}$) The kinetic admissible class $\mathcal{A}_{\textup{c.m.}}$ is the collection of kinetic distribution sequences $\mathbf{f}\in \ell^1(\mathbb{N};L^1(\omega\times\mathbb{R}^2))$ such that for all $j\in\mathbb{N}$ and all $(y,v)\in\omega\times\mathbb{R}^2$, 
$$0\le f_j(y,v) \leq 1\quad\textup{(Pauli exclusion principle)}.$$
\item (Quantum admissible class $\mathcal{A}_{\textup{q.m.}}$)
The quantum admissible class $\mathcal{A}_{\textup{q.m.}}$ is defined as the collection of quantum state sequences $\bchi$ such that for all $y\in\omega$, $\chi_j(y,\cdot)\in H_0^1(0,1)$ and 
$$\langle\chi_j(y,\cdot), \chi_k(y,\cdot)\rangle_{L^2(0,1)} =\delta_{jk} \textup{ for all }j,k\in\mathbb{N}\quad\textup{(partial orthonormality)}$$
\item (Admissible pair) A pair $(\mathbf{f},\bchi)\in\mathcal{A}_{\textup{c.m.}}\times\mathcal{A}_{\textup{q.m.}}$ is called admissible.
\end{enumerate} 
\end{definition}

\begin{assumption}[Assumptions on  the external potential and Casimir functionals]\label{assump on beta} \ 
\begin{enumerate}[$(i)$]
\item $V_{\textup{ext}}\in C(\overline{\Omega})\cap C^1(\Omega)$ and $V_{\textup{ext}}\ge 0$.
\item $\beta:[0,1]\to[0,\infty)$ is strictly convex and differentiable, and $\beta(0)=0$.
\end{enumerate}
\end{assumption} 

Now, we fix $M>0$ and $T\ge 0$, and we consider the free energy minimization problem under a mass-constraint, 
$$\boxed{\quad\mathcal{F}_{\textup{min}}(M)=\mathcal{F}_{\textup{min}}(M; T,\beta):=\inf\Big\{\mathcal{F}_{T,\beta}(\mathbf{f},\bchi): (\mathbf{f},\bchi)\textup{ is admissible and }\mathcal{M}(\mathbf{f}, \bchi)=M\Big\}.\quad}$$

\begin{remark}\label{remark: zero temperature remark}
$(i)$ In the literature, the generalized free energy minimization is considered to investigate thermal effects on quantum mixed states \cite{ADS, DFL}, but it is also used to construct a large class of stable stationary states \cite{GuoRein, Rein}.\\
$(ii)$ For the variational problem $\mathcal{F}_{\textup{min}}(M; T,\beta)$, the zero temperature case $T=0$ is included and it will be treated at the same time unless there is a confusion. Note that $\mathcal{F}_{\textup{min}}(M; 0,\beta)$ is simply the energy minimization problem.
\end{remark}

For the statement, we introduce 
\begin{equation}\label{beta tilde}
\tilde{\beta}(s):=\left\{\begin{aligned}
\mathbbm{1}_{\{s\ge 0\}}(s) &\textup{ if } T=0,\\
\mathbbm{1}_{\{s\ge 0\}}(s)  \inf\{(\beta')^{-1}(\tfrac{s}{T}),1\} &\textup{ if } T>0.
\end{aligned}\right.\end{equation}
Our first main result constructs free energy minimizers and shows their basic properties. 

\begin{theorem}[Minimization of the free energy]\label{existence of a mini} For $M>0$ and $T\ge 0$, the following hold.
\begin{enumerate}[$(1)$]
\item  (Existence) The variational problem $\mathcal{F}_{ \textup{min}}(M; T,\beta)$ has a minimizer $(\mathbf{f}^*, \bchi^*)\in\mathcal{A}_{\textup{c.m.}} \times\mathcal{A}_{\textup{q.m.}}$.
\item (Uniqueness) The minimizer $(\mathbf{f}^*, \bchi^*)$ is unique in the sense that if $(\tilde{\mathbf{f}}^*, \tilde{\bchi}^*)\in\mathcal{A}_{\textup{c.m.}} \times\mathcal{A}_{\textup{q.m.}}$ is another minimizer for $\mathcal{F}_{ \textup{min}}(M)$, then $U_{\rho_{(\mathbf{f}^*, \bchi^*)}}=U_{\rho_{(\tilde{\mathbf{f}}^*, \tilde{\bchi}^*)}}$.
\item (Self-consistent equation) The pair $(\mathbf{f}^*, \bchi^*)$ solves the self-consistent equation 
\begin{equation}\label{self consi}
f_j^*(y,v)=\tilde{\beta}\left(\mu -\frac{|v|^2}{2}-\lambda_j^*(y)\right)
\end{equation}
for some $\mu>0$ and the Schr\"odinger-Poisson system
\begin{equation}\label{Schro Poisson}
\left(-\frac{\partial_z^2}{2}+ (U_{\rho_{(\mathbf{f}^*,\bchi^*)}}+V_{\textup{ext}})(y,\cdot)\right)\chi_j^*(y,\cdot)=\lambda_j^*(y) \chi_j^*(y,\cdot),
\end{equation}
where $\lambda_j^*(y)$ is the $j$-th eigenvalue for the Schr\"odinger operator $-\frac{\partial_z^2}{2}+ (U_{\rho_{(\mathbf{f}^*,\bchi^*)}}+V_{\textup{ext}})(y,\cdot)$ with form domain $H_0^1(0,1)$. 
\item (Monotone finite subband structure) For a.e. $(y,v)\in\omega\times\mathbb{R}^2$, $f_j^*(y,v)$ is strictly decreasing in $j$ and $f_j^*(y,v)\equiv0$ for $j\geq\frac{\sqrt{3\mu}}{\pi}$.
\end{enumerate}
\end{theorem}

\begin{remark}
In Ben Abdallah-M\'ehats \cite[Theorem 5.1]{BM}, a different class of steady states are constructed. The main contribution of Theorem \ref{existence of a mini} is to provide steady states obeying the self-consistent equation \eqref{self consi}. A key finding is that this class of minimizers allows only finitely many non-trivial kinetic distributions which are decreasing in $j$ pointwisely (see Theorem \ref{existence of a mini} (4)). Such properties are non-trivial and distinctive from general states under consideration.
\end{remark}

 \begin{remark}
In Theorem \ref{existence of a mini} (2), uniqueness should be stated in terms of potential functions. Indeed, given a minimizer, more minimizers can be constructed rearranging $(\mathbf{f}^*, \bchi^*)$ in $j$ (see Section \ref{sec: rearrangement of admissible pairs}). Moreover, since the mass and the free energy do not distinguish quantum states outside the support of density functions, replacing a quantum state $\chi_j^*$ outside the support of $\rho_{f_j^*}$, different minimizers are freely generated. However, all such minimizers have the same potential function $U_{\rho_{(\mathbf{f}, \bchi)}}$.
\end{remark}

Next, we establish the dynamical stability of minimizers.

\begin{theorem}[Conditional dynamical stability]\label{stability}
For $M>0$  and $T\ge 0$, a unique minimizer $(\mathbf{f}^*,\bchi^*)$ for the problem $\mathcal{F}_{ \textup{min}}(M; T,\beta)$ constructed in Theorem  \ref{existence of a mini} is a stable solution to the Vlaosv-Schr\"odinger-Poisson system in the following sense: Given $\epsilon>0$, there exists $\delta>0$ such that the following hold. We assume that 
\begin{enumerate}[$(i)$]
\item $(\mathbf{f}_0,\bchi_0)\in\mathcal{A}_{\textup{c.m.}}\times\mathcal{A}_{\textup{q.m.}}$, $|\mathcal{M}(\mathbf{f}_0)-M|\leq \delta$ and $|\mathcal{F}(\mathbf{f}_0,\bchi_0)-\mathcal{F}(\mathbf{f}^*,\bchi^*)|\leq\delta$.
\item $(\mathbf{f}(t),\bchi(t))$ is a unique global weak solution to the Vlasov-Schr\"odinger-Poisson system \eqref{Vlasov equation}-\eqref{Poisson equation} with initial data $(\mathbf{f}_0,\bchi_0)$, and $\mathbf{f}(t)$ satisfies the the specular-reflection boundary condition \eqref{specular-reflection boundary condition}. Moreover, $\mathcal{M}(\mathbf{f}(t),\bchi(t))=\mathcal{M}(\mathbf{f}_0,\bchi_0)$ and $\mathcal{F}(\mathbf{f}(t),\bchi(t))=\mathcal{F}(\mathbf{f}_0,\bchi_0)$ for all $t\in\mathbb{R}$.
\end{enumerate}
Then,
$$\sup_{t\in\mathbb{R}}\|\nabla (U_{\rho_{(\mathbf{f}(t),\bchi(t))}}-U_{\rho_{(\mathbf{f}^*,\bchi^*)}})\|_{L^2(\Omega)}^2\leq\epsilon.$$
\end{theorem}

\begin{remark}
The stability in Theorem \ref{stability} is \textit{conditional}, because well-posedness of the initial-value problem \eqref{Vlasov equation}-\eqref{Poisson equation} is currently unknown. For this reason, the assumption $(ii)$ is imposed. Even while it is conditional, to the best of the authors' knowledge, Theorem \ref{stability} is the first stability result for subband models.
\end{remark}

\subsection{Idea of the proof}

Our main results consist of construction of free energy minimizers for the subband model, and their uniqueness and stability properties. Here, we describe the main difficulties coming from the hybrid nature, and present the outline of the proof. 

First, we note that the variational $\mathcal{F}_{\textup{min}}(M)=\mathcal{F}_{\textup{min}}(M; T,\beta)$ has a minimizing sequence $\{(\mathbf{f}^{(n)},\bchi^{(n)})\}_{n=1}^\infty\subset \mathcal{A}_{\textup{c.m.}}\times\mathcal{A}_{\textup{q.m.}}$ each of which has finite free energy. Indeed, $\mathcal{F}_{\textup{min}}(M)\geq0$, and an upper bound for $\mathcal{F}_{\textup{min}}(M)$ is obtained by a smooth single-band pair $(\tilde{\mathbf{f}}, \tilde{\bchi})$ such that $\|\tilde{f}_1\|_{L^1(\omega\times \R^2)}=M$ but $\tilde{f}_j\equiv0$ for all $j\geq 2$. Thus, one may attempt to construct a minimizer as a (weak subsequential) limit of a minimizing sequence. Then, one would however encounter several non-trivial obstacles described below. 

\begin{remark}[Difficulties in the variational problem $\mathcal{F}_{\textup{min}}(M)$]\label{difficulties}\
\begin{enumerate}[$(i)$]
\item (Lack of compactness with respect to the $j$-index) All $j$-summed quantities in the mass and the free energy are invariant under translation in $j$. Hence, the dichotomy and the vanishing (in $j$) scenarios cannot be eliminated in general. This issue may be avoided restricting the number of non-trivial subbands\footnote{For $J\in\mathbb{N}$, the admissible class is further restricted to kinetic distributions with $f_j\equiv 0$ for all $j\geq J+1$}, but its minimizer (even if it exists) might be unstable under general perturbations with more non-zero subbands.

\item (Compactness of the quantum state part) More significant challenges arise in the quantum states $\{\bchi^{(n)}\}_{n=1}^\infty$. At first glance, in the free energy (see \eqref{free energy}), the term 
$$\sum_{j=1}^\infty \iint_{\omega \times\mathbb{R}^2} \|\partial_z\chi_j^{(n)}(y,\cdot)\|_{L^2(0,1)}^2f_j^{(n)}(y,v)dydv$$
seems the only useful term for compactness of $\{\bchi^{(n)})\}_{n=1}^\infty$. However, it cannot be directly employed due to the following two reasons.
\begin{enumerate}[$(a)$]
\item A uniform bound on $\|\partial_z\chi_j^{(n)}(y,\cdot)\|_{L^2(0,1)}$ is difficult to obtain, because the weight function $f_j^{(n)}(y,v)$ also depends on $n$ and its uniform lower bound is not expected.
\item If $\|\partial_z\chi_j^{(n)}(y,\cdot)\|_{L^2(0,1)}$ is shown to be bounded uniformly in $n$, for each $y\in\omega$, the Rellich-Kondrachov theorem provides convergence of a subsequence $\{\chi_{j_{k;y}}^{(n)}(y,\cdot)\}_{n=1}^\infty$, where the choice of a sub-index $\{j_{k;y}\}_{k=1}^\infty$ depends on $y\in\omega$. Thus, as functions of $y$, a convergent subsequence of $\{\chi_j^{(n)}(y,\cdot)\}_{n=1}^\infty$ cannot be chosen. This is due to the quantum-kinetic hybrid nature of the problem where no good derivative norm bound is available in the unconfined $y$-variable direction. 
\end{enumerate} 
\end{enumerate}
\end{remark}

We solve the above problems by refining minimizing sequences in three steps.  
\begin{enumerate}[(\textup{Step }1)]
\item The mass and the free energy are invariant under the rearrangements in Section \ref{sec: rearrangement of admissible pairs}. Thus, rearranging each $(\mathbf{f}^{(n)},\bchi^{(n)})$ with non-decreasing quantum kinetic energies, we can derive a uniform bound on the weighted summation norm for $\mathbf{f}^{(n)}$ (see Proposition \ref{prop: refinement 1}), and it overcomes the obstacle $(i)$. 
\item We consider the partial free energy minimization ``fixing $\bchi^{(n)}$  from Step 1." By the concentration-compactness principle with the uniform weighted summation norm bound in Step 1, we construct a minimizer $\tilde{\mathbf{f}}^{(n)}$ solving a self-consistent equation. Then, we refine the sequence replacing $\mathbf{f}^{(n)}$ by $\tilde{\mathbf{f}}^{(n)}$. Consequently, we deduce from the self-consistent equation   that the refined minimizing sequence has only finitely many non-trivial subbands and kinetic distributions are non-increasing in $j$ (see Proposition \ref{prop: refinement 2}).
\item We consider another partial free energy minimization ``fixing $\mathbf{f}^{(n)}$ from Step 2." Here, we do not employ the concentration-compactness principle. Instead, we take the unique solution $\tilde{\bchi}^{(n)}$ to the Schr\"odinger-Poisson equation \eqref{Schrodinger equation}-\eqref{Poisson equation} with $\mathbf{f}=\mathbf{f}^{(n)}$, constructed in Ben Abdallah-M\'ehats \cite{BM}, which exists because $\{f_j^{(n)}\}_{j=1}^\infty$ is non-increasing in $j$. Then, using coercivity of the free energy functionals (Lemma \ref{lemma: energy coercivity}), we show that $\tilde{\bchi}^{(n)}$ is a desired minimizer. Moreover, replacing $\bchi^{(n)}$ by $\tilde{\bchi}^{(n)}$, a minimizing sequence is refined so that each $\bchi^{(n)}$ solves the Schr\"odinger equation \eqref{Schrodinger equation} (see Proposition \ref{proposition: reduced problem}).
\end{enumerate}
By Step 1-3, a minimizing sequence $(\mathbf{f}^{(n)},\bchi^{(n)})$ is refined so that it has only finitely many non-trivial subbands, distributions in $\mathbf{f}^{(n)}$ is non-increasing, and $\bchi^{(n)}$ are eigenfunctions of the quasi-static 1D Schr\"odinger operator. Then, it follows that the nonlinear potentials, $U^{(n)}=U_{\rho_{(\mathbf{f}^{(n)},\bchi^{(n)})}}+V_{\textup{ext}}$'s, are uniformly bounded and equicontinuous with respect to $x\in\Omega$ and $n\geq 1$, and so are eigenfunctions $\bchi^{(n)}=\{\chi_{j}^{(n)}(x)\}_{j=1}^\infty$ to $-\frac{\partial_z^2}{2}+U^{(n)}$. Therefore, finally, we use the Arzel\`a-Ascoli theorem to take a sub-sequential limit of the refined minimizing sequence. 
 
For uniqueness and stability, the coercivity lemma for the free energy (Lemma \ref{lemma: energy coercivity}) is also crucially employed together with the structure of the minimizer from the self-consistent equation. 

\subsection{Organization of the paper}
In Section \ref{sec: Basic properties of total densities}, we provide a basic kinetic interpolation inequality for the Vlasov-Schr\"odinger-Poisson system. In the next three sections (Section \ref{sec: 1st refinement}-\ref{sec: 3rd refinement}), we refine a minimizing sequence in three steps as described above. Finally, in Section \ref{exist and uniq stable}, using the refined minimizing sequence, we prove the main results.

\subsection{Acknowledgement}
This work was supported by National Research Foundation of Korea (NRF) grant funded by the Korean government (MSIT) (No. NRF-2020R1A2C4002615). The authors would like to thank Prof. Jin Woo Jang at Postech for kindly explaining various boundary conditions in kinetic theory.

\section{Basic inequalities for total densities}\label{sec: Basic properties of total densities}

To begin with, we present elementary inequalities for total densities (see \eqref{total density function}) which will be applied to the variational problem. The following is a simple extension of the well-known interpolation inequalities in kinetic theory.

\begin{lemma}[Kinetic interpolation inequality]\label{basic estimates}
Let $s \in[1,\infty]$. Then, for an admissible pair $(\mathbf{f},\bchi)\in \mathcal{A}_{\textup{c.m.}}\times\mathcal{A}_{\textup{q.m.}}$, we have
\begin{equation}\label{kinetic interpolation inequality}
\begin{aligned}
\|\rho_{(\mathbf{f}, \bchi)}\|_{L^\frac{5s-3}{3s-1}(\Omega)}&\lesssim \|\mathbf{f}\|_{\ell^1(\mathbb{N}; L^s(\omega\times \R^2))}^\frac{2s}{5s-3}\big\||v|^2\mathbf{f}\big\|_{\ell^1(\mathbb{N}; L^1(\omega\times \R^2))}^\frac{2(s-1)}{5s-3}\\
&\quad\cdot\bigg\{ \sum_{j=1}^\infty\big\|\|\partial_z\chi_j\|_{L^2(0,1)}^2f_j\big\|_{L^1(\omega\times\mathbb{R}^2))}\bigg\}^\frac{s-1}{5s-3}.
\end{aligned}
\end{equation}
\end{lemma}

\begin{proof}
The proof can be reduced to the single subband case $f_j\equiv 0$ for $j\geq 2$, because the inequality \eqref{kinetic interpolation inequality} simply follows from the single subband inequality to the right hand side of $\|\rho_{(\mathbf{f}, \bchi)}\|_{L^\frac{5s-3}{3s-1}(\Omega)}\leq\sum_{j=1}^\infty\|\rho_{f_j}|\chi_j|^2\|_{L^\frac{5s-3}{3s-1}(\Omega)}$ with the H\"older inequality in the sum $\sum_{j=1}^\infty$.

In the single subband case, applying the 1D Gagliardo-Nirenberg inequality 
$$\|\chi\|_{L^{\frac{2(5s-3)}{3s-1}}(0,1)}\lesssim \|\chi\|_{L^{2}(0,1)}^\frac{2(2s-1)}{5s-3}\|\partial_z\chi\|_{L^2(0,1)}^\frac{s-1}{5s-3}$$
with $\|\chi_1\|_{L^{2}(0,1)}=1$ and the H\"older inequality, we obtain 
$$\begin{aligned}
\|\rho_{f_1}|\chi_1|^2\|_{L^\frac{5s-3}{3s-1}(\Omega)}&=\left\|\rho_{f_1}\|\chi_1\|_{L^{\frac{2(5s-3)}{3s-1}}(0,1)}^2\right\|_{L^\frac{5s-3}{3s-1}(\omega)}\\
&\lesssim\big\|\rho_{f_1}^\frac{2(2s-1)}{5s-3} \rho_{f_1}^\frac{s-1}{5s-3} \|\partial_z\chi_1\|_{L^2(0,1)}^\frac{2(s-1)}{5s-3}\big\|_{L^\frac{5s-3}{3s-1}(\omega)}\\ 
&\lesssim\|\rho_{f_1}\|_{L^\frac{2s-1}{s}(\omega)}^\frac{2(2s-1)}{5s-3}\big\|\rho_{f_1} \|\partial_z\chi_1\|_{L^2(0,1)}^2\big\|_{L^1(\omega)}^\frac{s-1}{5s-3}.
\end{aligned}$$
Finally, applying the well-known interpolation inequality
\begin{equation}\label{inter ineq1}
\|\rho_f\|_{L^\frac{2s-1}{s}(\omega)}\lesssim\|f\|_{L^s(\omega\times \R^2)}^\frac{s}{2s-1}\big\||v|^2f\big\|_{L^1(\omega\times \R^2)}^\frac{s-1}{2s-1},
\end{equation}
we complete the proof.
\end{proof}

For a general admissible pair $(\mathbf{f},\bchi)$, the inequality \eqref{kinetic interpolation inequality} is not good enough to control its total density function, because $\|\mathbf{f}\|_{\ell^1(\mathbb{N};L^s(\omega\times \R^2))}$ is not a priori-ly bounded. Nevertheless, we will be concerned only with admissible pairs uniformly bounded in a weighted norm (see Proposition \ref{prop: refinement 1}). For them, we have a good bound. 

\begin{lemma}[Upgraded total density estimate]\label{density estimate}
Let $T\geq 0$. For $(\mathbf{f},\bchi)\in \mathcal{A}_{\textup{c.m.}}\times\mathcal{A}_{\textup{q.m.}}$, we have 
$$\|\rho_{(\mathbf{f}, \bchi)}\|_{L^\frac{5s-3}{3s-1}(\Omega)}\lesssim  \bigg\{\sum_{j=1}^\infty j^2\|f_j\|_{L^1(\omega\times\mathbb{R}^2)}\bigg\}^\frac{2}{5s-3}\mathcal{F}(\mathbf{f}, \bchi)^\frac{3(s-1)}{5s-3}\quad \textup{for any }s\in[1,3).$$
\end{lemma}

\begin{proof}
  It is enough to show that all factors on the right hand sides of \eqref{kinetic interpolation inequality} are bounded by $\mathcal{F}(\mathbf{f}, \bchi)$ and $\sum_{j=1}^\infty j^2\|f_j\|_{L^1(\omega\times\mathbb{R}^2)}$. Indeed, the second and the last factors are obviously bounded by $\sim  \mathcal{F}(\mathbf{f}, \bchi)$. Hence, we only consider $\|\mathbf{f}\|_{\ell^1(\mathbb{N}; L^s(\omega\times \R^2))}$. Note that since $0\leq f_j\leq 1$, we have $\|\mathbf{f}\|_{\ell^1(\mathbb{N}; L^s(\omega\times \R^2))}=\sum_{j=1}^\infty \|f_j\|_{L^s(\omega\times\mathbb{R}^2)}\leq \sum_{j=1}^\infty \|f_j\|_{L^1(\omega\times\mathbb{R}^2)}^{1/s}$. Hence, by the H\"older inequality in $\sum_{j=1}^\infty$ with $\frac{2}{s}\cdot\frac{s}{s-1}>1$, we obtain
\begin{equation}\label{inter ineq2}
\|\mathbf{f}\|_{\ell^1(\mathbb{N}; L^s(\omega\times \R^2))}\leq \sum_{j=1}^\infty j^{-\frac{2}{s}} \Big(j^2\|f_j\|_{L^1(\omega\times\mathbb{R}^2)}\Big)^{\frac{1}{s}}\lesssim\bigg\{\sum_{j=1}^\infty j^2\|f_j\|_{L^1(\omega\times\mathbb{R}^2)}\bigg\}^{\frac{1}{s}}.
\end{equation}
Thus, combining these, we prove the result.
\end{proof}

\section{First refinement of a minimizing sequence: improved  summability}\label{sec: 1st refinement}

In this section, we rearrange a minimizing sequence for better summability.

\begin{proposition}[Refined minimizing sequence with a better summability]\label{prop: refinement 1}
For $M>0$  and $T\ge 0$, the variational problem $\mathcal{F}_{\textup{min}}(M)$ has a minimizing sequence
$$\{(\mathbf{f}^{(n)},\bchi^{(n)})\}_{n=1}^\infty\subset \mathcal{A}_{\textup{c.m.}}\times\mathcal{A}_{\textup{q.m.}}^\uparrow$$
where $\mathcal{A}_{\textup{q.m.}}^\uparrow\subset \mathcal{A}_{\textup{q.m.}}$ is the subclass of quantum admissible states having non-decreasing kinetic energies, i.e., for each $y\in\omega$, $\|\partial_z\chi_j(y,\cdot)\|_{L^2(0,1)}$ is non-decreasing in $j\geq1$. As a consequence (by Lemma \ref{improved summability} below), we have
\begin{equation}\label{weighted ell 1 bound}
\sup_{n\geq 1}\left\{\sum_{j=1}^\infty j^2\|f_j^{(n)}\|_{ L^1(\omega\times\mathbb{R}^2)}\right\}\lesssim \mathcal{F}_{\textup{min}}(M),
\end{equation}
\end{proposition}

\begin{remark}\label{prop: refinement 1, remark}
If $\bchi\in\mathcal{A}_{\textup{q.m.}}^\uparrow$, then $\|\partial_z\chi_j(y,\cdot)\|_{L^2(0,1)}\to\infty$ as $j\to\infty$, because $-\partial_z^2$ is unbounded as a quadratic form on any infinite-dimensional subspace in $H_0^1(0,1)$. Indeed, for all $j\geq 1$, a precise lower bound $\|\partial_z\chi_j(y,\cdot)\|_{L^2(0,1)}\geq \tfrac{\pi j}{\sqrt{3}}$ can be obtained (see \eqref{improved summability proof}).
\end{remark}

\begin{remark}
The uniform weighted $\ell^1(\mathbb{N})$-bound \eqref{weighted ell 1 bound} resolves the lack of compactness in $j$ mentioned in Remark \ref{difficulties} $(i)$ (see the proof of Lemma \ref{reduced problem} below).
\end{remark}

\subsection{Rearrangement of admissible pairs}\label{sec: rearrangement of admissible pairs}
We note that any permutation-valued function, with the phase space $\omega\times \R^2$, preserves the mass and the free energy. Precisely, let
$$\sigma=\sigma(j;y,v):\mathbb{N}\times\omega\times \R^2\to\mathbb{N}$$
be a function such that $\sigma=\sigma(\cdot;y,v):\mathbb{N}\to\mathbb{N}$ is a bijection for each $(y,v)\in\omega\times \R^2$. For an admissible pair $(\mathbf{f},\bchi)$, we define the $\sigma$-rearranged pair $(\mathbf{f}^\sigma,\bchi^\sigma)$ by 
\begin{equation}\label{rearrang def}
\big(f_j^\sigma(y,v), \chi_j^\sigma(x)\big)=\big(f_{\sigma(j;y,v)}(y,v),\chi_{\sigma(j;y,v)}(x)\big).
\end{equation}
Then, one can easily see that the rearranged pair is also admissible, and that the mass and the energy are invariant under the rearrangement, i.e.,
$$\mathcal{M}(\mathbf{f}^\sigma)=\mathcal{M}(\mathbf{f})\quad\textup{and}\quad\mathcal{F}(\mathbf{f}^\sigma,\bchi^\sigma)=\mathcal{F}(\mathbf{f},\bchi),$$
because the sums in the mass and the free energy functional $\sum_{j=1}^\infty f_j$, $\sum_{j=1}^\infty \beta (f_j ), \cdots$  $\sum_{j=1}^\infty \|\partial_z\chi_j\|_{L^2(0,1)}^2 {f_j}$ and $\rho_{(\mathbf{f}, \bchi)}$  are invariant under the rearrangement.

\subsection{Improved summability by rearrangement}

For Proposition \ref{prop: refinement 1}, we employ the following   rearrangements. 

\begin{lemma}[Rearrangement with non-decreasing quantum kinetic energies]\label{lemma: rearranged pair}
For $\bchi\in \mathcal{A}_{\textup{q.m.}}$, there exists $\sigma_\uparrow:\mathbb{N}\times\omega\to\mathbb{N}$ such that $\bchi^{\sigma_\uparrow}\in \mathcal{A}_{\textup{q.m.}}^\uparrow$.
\end{lemma}
\begin{proof}
For $y\in\omega$, we have $\|\partial_z\chi_j(y,\cdot)\|_{L^2(0,1)}\to\infty$, 
because $-\partial_z^2$ is unbounded on any infinite-dimensional subspace in $H_0^1(0,1)$. Thus, there exists a bijection ${\sigma_\uparrow}={\sigma_\uparrow}(\cdot;y):\mathbb{N}\to\mathbb{N}$ such that $\|\partial_z\chi_j^{\sigma_\uparrow}(y,\cdot)\|_{L^2(0,1)}$ is non-decreasing in $j$ and goes to infinity (see \eqref{rearrang def} for the definition of $\chi_j^{\sigma_\uparrow}$).
\end{proof}

If a pair is rearranged by Lemma \ref{lemma: rearranged pair}, then it enjoys a better summability in $j$. 

\begin{lemma}[Weighted $\ell^1(\mathbb{N})$-bound]\label{improved summability}
If $ (\mathbf{f}, \bchi) \in \mathcal{A}_{\textup{c.m.}}\times\mathcal{A}_{\textup{q.m.}}^\uparrow$,  then it holds that
$$
\sum_{j=1}^\infty j^2\|f_j\|_{L^1(\omega\times\mathbb{R}^2)}\leq\frac{3}{\pi^2}\sum_{j=1}^\infty\int_\omega \|\partial_z\chi_j(y,\cdot)\|_{L^2(0,1)}^2\rho_{f_j}(y)dy  \leq\frac{6}{\pi^2}  \mathcal{F}(\mathbf{f},\bchi).$$ 
\end{lemma}

\begin{proof} 
It suffices to show that for all $j\geq1$ and all $y\in\omega$,
\begin{equation}\label{improved summability proof}
\|\partial_z\chi_j(y,\cdot)\|_{L^2(0,1)}\geq \tfrac{\pi j}{\sqrt{3}},
\end{equation}
because applying \eqref{improved summability proof} to the free energy \eqref{free energy}, the lemma follows. Indeed, we assume that there exist $\tilde{y}\in\omega$ and $\tilde{j}\geq 1$ such that $\|\partial_z\chi_{\tilde{j}}(\tilde{y},\cdot)\|_{L^2(0,1)}< \frac{\pi \tilde{j}}{\sqrt{3}}$. Then, since $\|\partial_z\chi_j(\tilde{y},\cdot)\|_{L^2(0,1)}$ does not decrease, we have $\sum_{j=1}^{\tilde{j}}\|\partial_z\chi_j(\tilde{y},\cdot)\|_{L^2(0,1)}^2<\frac{\pi^2(\tilde{j})^3}{3}$. On the other hand, the min-max principle \cite[Theorem 12.1]{LiebLoss} implies
$$\sum_{j=1}^{\tilde{j}}\big\langle (-\partial_z^2)\chi_j(\tilde{y},\cdot),\chi_j(\tilde{y},\cdot)\big\rangle_{L^2(0,1)}\geq\sum_{j=1}^{\tilde{j}}(\pi j)^2=\frac{\pi^2\tilde{j}(\tilde{j}+1)(2\tilde{j}+1)}{6},$$
which deduces a contradiction.
\end{proof}

\subsection{Proof of Proposition \ref{prop: refinement 1}}
Rearranging a minimizing sequence by Lemma \ref{lemma: rearranged pair} but still denoting by $(\mathbf{f}^{(n)},\bchi^{(n)})$, we may assume that 
 $\bchi^{(n)}\in \mathcal{A}_{\textup{q.m.}}^\uparrow$. Then, the weighted norm bound \eqref{weighted ell 1 bound} follows from Lemma \ref{improved summability}.

\subsection{Monotone kinetic distributions by rearrangement}
Later, we employ another rearrangement to make kinetic distributions non-increasing. This rearrangement also preserves the mass and the free energy.

\begin{lemma}[Rearrangement with non-increasing  kinetic distributions]\label{rearrange distrib}
For $\mathbf{f}\in \mathcal{A}_{\textup{c.m.}}$, there exists $\sigma_\downarrow:\mathbb{N}\times\omega\times \R^2\to\mathbb{N}$ such that $f_j^{\sigma_\downarrow}(y,v)$ is non-increasing in $j\geq 1$ for a.e. $(y,v)\in\omega\times\mathbb{R}^2$.
\end{lemma}
\begin{proof}
For a.e. $(y,v)\in\omega \times \R^2$, we have $f_j(y,v) \to0$ as $j\to \infty$, 
because $\mathcal{M}(\mathbf{f})<\infty$. Thus, there exists a bijection $\sigma_\downarrow=\sigma_\downarrow(\cdot;y,v): \mathbb{N}\to\mathbb{N}$ such that $f_j^{\sigma_\downarrow}(y,v) $ is non-increasing in $j$ (see \eqref{rearrang def} for the definition of $f_j^{\sigma_\downarrow}$).
\end{proof}

\section{Second refinement of a minimizing sequence: finite non-increasing band structure}

In the previous section, a minimizing sequence $\{(\mathbf{f}^{(n)},\bchi^{(n)})\}_{n=1}^\infty$ is refined so that the kinetic distributions $\{\mathbf{f}^{(n)}\}_{n=1}^\infty$ has a weak subsequential limit, while the difficulties mentioned in Remark \ref{difficulties} $(ii)$ have not been overcome for the quantum part $\{\bchi^{(n)}\}_{n=1}^\infty\subset\mathcal{A}_{\textup{q.m.}}^\uparrow$.

In this section, fixing $\bchi\in\mathcal{A}_{\textup{q.m.}}^\uparrow$, we consider the   ``partial" variational problem
\begin{equation}\label{free energy minimization fixing bchi}
\mathcal{F}_{\bchi;\textup{min}}(M)=\inf\Big\{\mathcal{F}_{\bchi}(\mathbf{f})=\mathcal{F}(\mathbf{f},\bchi):\ \mathbf{f}\in\mathcal{A}_{\textup{c.m.}}\textup{ and }\mathcal{M}(\mathbf{f})=M\Big\}.
\end{equation}
Then, we further refine the minimizing sequence in Proposition \ref{prop: refinement 1} by replacing $\mathbf{f}^{(n)}$ by a minimizer for the problem $\mathcal{F}_{\bchi^{(n)};\textup{min}}(M)$. We assert that this secondly refined minimizing sequence has only finitely many non-trivial subbands with preferable monotonicity.

\begin{proposition}[Refined minimizing sequence with finite subbands]\label{prop: refinement 2}
For $M>0$  and $T\ge 0$, the variational problem $\mathcal{F}_{\textup{min}}(M)$ has a minimizing sequence $$\{(\mathbf{f}^{(n)},\bchi^{(n)})\}_{n=1}^\infty\subset \mathcal{A}_{\textup{c.m.}}^\downarrow\times\mathcal{A}_{\textup{q.m.}},$$
where $\mathcal{A}_{\textup{c.m.}}^\downarrow$ is the subclass of non-increasing kinetic admissible distributions, i.e., for a.e. $(y,v)\in\omega\times\mathbb{R}^2$, $f_j(y,v)$ is non-increasing in $j\geq 1$. Moreover, it satisfies the following.
\begin{enumerate}[$(i)$]
\item (Finite subband structure) There exists $J\geq1$, independent of $n\geq 1$, such that $f_j^{(n)}\equiv0$ for all $j\geq J+1$.
\item (Bounded densities) $\rho_{f_j^{(n)}}(y)$ are bounded uniformly in $y\in\omega$ and $j,n\geq1$.
\end{enumerate}
\end{proposition}

\subsection{Free energy minimization for fixed quantum states}
First, we show existence of a minimizer for the simpler partial problem.

\begin{lemma}[Existence of a minimizer for the partial variational problem $\mathcal{F}_{\bchi;\textup{min}}(M)$]\label{reduced problem}
For $T\ge 0$,  $M>0$ and $\bchi\in\mathcal{A}_{\textup{q.m.}}^\uparrow$, the problem $\mathcal{F}_{\bchi;\textup{min}}(M)$ has a unique minimizer $\mathbf{f}^*$.
\end{lemma}

For the proof, we recall that the weak formulation of the Poisson equation \eqref{Poisson equation} is stated as 
\begin{equation}\label{weak Poisson}
\int_{\Omega}\nabla U\cdot\nabla \varphi dx=\int_{\Omega}\rho \varphi dx\quad\textup{for all }\varphi\in H_\omega^1(\Omega),
\end{equation}
where
$$H_\omega^1(\Omega):=\Big\{u\in H^1(\Omega): u|_{\partial\omega\times(0,1)}=0 \Big\}.$$
By the the Lax-Milgram theorem and the Poincar\'e inequality, a density $\rho\in L^{\frac{6}{5}}(\Omega)$ uniquely determines a potential $U_\rho\in H_\omega^1(\Omega)$ solving the equation \eqref{weak Poisson}. 
Thus, we have the following.
\begin{lemma}[Potential function]\label{potential function}
The density-to-potential map $\rho\mapsto U_\rho: L^{\frac{6}{5}}(\Omega)\to H_\omega^1(\Omega)$ is linear and bounded.
\end{lemma}

\begin{proof}[Proof of Lemma \ref{reduced problem}]
It suffices to show existence, because the convexity of the functional $\mathcal{F}_{\bchi}(\mathbf{f})$ implies uniqueness. Let $\{\mathbf{f}^{(n)}\}_{n=1}^\infty$ be a minimizing sequence for $\mathcal{F}_{\bchi;\textup{min}}(M)$. Then, passing to a subsequence, we may assume that for each $j$, ${f}_j^{(n)}\rightharpoonup f_j^*$ in $L^r(\omega\times \R^2)$ for any $1<r<\infty$ so that $\mathbf{f}^*\in \mathcal{A}_{\textup{c.m.}}$. We will that $\mathbf{f}^*$ is a minimizer for $\mathcal{F}_{\bchi;\textup{min}}(M)$.

Note that by \cite[Corollary 3.9]{Brezis} and the positive terms in the series, the following quantities are reduced in the weak limit;
$$\liminf_{n\to \infty}\sum_{j=1}^\infty\||v|^2 {f}_j^{(n)}\|_{L^1(\omega \times \R^2)}\ge \sum_{j=1}^\infty \||v|^2 {f}_j\|_{L^1(\omega \times \R^2)},\quad \liminf_{n\to \infty}\mathcal{C}_\beta(\mathbf{f}^{(n)}) \geq \mathcal{C}_\beta(\mathbf{f}^*)$$
and
\begin{equation}\label{quantum energy lower bound}
\begin{aligned}
&\liminf_{n\to \infty}\sum_{j=1}^\infty\iint_{\omega\times\mathbb{R}^2} \langle (-\tfrac{\partial_z^2}{2}+V_{\textup{ext}}) \chi_j, \chi_j\rangle_{L^2 (0,1)}{f}_j^{(n)}dydv\\
&\ge \sum_{j=1}^\infty\iint_{\omega\times\mathbb{R}^2}\langle (-\tfrac{\partial_z^2}{2}+V_{\textup{ext}}) \chi_j, \chi_j\rangle_{L^2 (0,1)}{f}_j^* dydv.
\end{aligned}
\end{equation}
In addition, by the uniform weighted norm bound (Lemma \ref{improved summability}), we have
\begin{equation}\label{weighted norm bound for Step 2}
\mathcal{F}_{\bchi; \textup{min}}(M)\gtrsim\liminf_{n\to \infty}\sum_{j=1}^\infty j^2\|f_j^{(n)}\|_{ L^1(\omega\times\mathbb{R}^2)}\geq \sum_{j=1}^\infty j^2\|f_j^*\|_{ L^1(\omega\times\mathbb{R}^2)}.
\end{equation}
Thus, except the nonlinear potential energy, all terms in the free energy are clearly reduced in the weak limit (see \eqref{energy} and \eqref{free energy}). For the nonlinear potential energy, we claim that as $n\to\infty$,
\begin{equation}\label{l3weak}
\rho_{(\mathbf{f}^{(n)}, \bchi)}\rightharpoonup\rho_{(\mathbf{f}^*, \bchi)} \textup{ in } L^\frac{6}{5}(\Omega).
\end{equation}
Indeed, by the density estimate (Lemma \ref{density estimate}), it suffices to show that for $g \in C_c^\infty(\omega)$, 
\begin{equation}\label{l1weak}
\lim_{n\to\infty}\int_\Omega (\rho_{(\mathbf{f}^{(n)}, \bchi)}- \rho_{(\mathbf{f}^*, \bchi)})gdy=0.
\end{equation}
For small $\epsilon>0$, by the weighted norm bound (Lemma \ref{improved summability}), taking large $J_\epsilon\geq1$, we obtain 
$$\int_\Omega (\rho_{(\mathbf{f}^{(n)}, \bchi)}-\rho_{(\mathbf{f}^*, \bchi)})g dx=\sum_{j=1}^{J_\epsilon}\int_\Omega (\rho_{f_j^{(n)}}-\rho_{f_j^*})|\chi_j|^2g dx+O(\epsilon).$$
On the other hand, we have
$$\int_\Omega (\rho_{f_j^{(n)}}-\rho_{f_j^*})|\chi_j|^2g dx=\int_{\omega\times\mathbb{R}^2} (f_j^{(n)}-f_j^*)\mathbbm{1}_{|v|\leq R}\left(\int_0^1|\chi_j|^2gdz\right) dx+O(R^{-2})\to 0$$
taking $n\to\infty$ and then $R\to\infty$. Hence, the claim \eqref{l3weak} is proved. Now, we observe from Lemma \ref{potential function} that the functional $\rho\mapsto \int_\Omega |\nabla U_\rho|^2dx$ is lower semi-continuous on $L^\frac65 (\Omega)\to \R$ in the strong topology.  Thus, by the weak convergence \eqref{l3weak} and the convexity of the map $\rho\mapsto \int_\Omega |\nabla U_\rho|^2dx$, we conclude that
$$\liminf_{n\to\infty}\|\nabla U_{(\mathbf{f}^{(n)},\bchi )}\|_{L^2(\Omega )}\ge\|\nabla U_{(\mathbf{f}^*,\bchi)}\|_{L^2(\Omega )}$$
holds (see \cite[Corollary 3.9]{Brezis}). Therefore, we conclude that 
$$\mathcal{F}_{\bchi;\textup{min}}(M)=\liminf_{n\to \infty}\mathcal{E}(\mathbf{f}^{(n)},\bchi)\ge \mathcal{E}(\mathbf{f},\bchi).$$

It remains to show that $\mathbf{f}^*=\{f_j^*\}_{j=1}^\infty$ is admissible for the problem $\mathcal{F}_{\bchi;\textup{min}}(M)$. Indeed, for any small $\epsilon>0$, by Lemma \ref{improved summability} and \eqref{weighted norm bound for Step 2}, there exist large $J_\epsilon, R_\epsilon\geq 1$ such that
$$\begin{aligned}
\mathcal{M}(\mathbf{f}^{(n)})&=\sum_{j=1}^{J_\epsilon} \iint_{\omega\times\mathbb{R}^2}\mathbbm{1}_{|v|\leq R_\epsilon}{f}_j^{(n)}dydv+O(\epsilon),\\
\mathcal{M}(\mathbf{f}^*)&=\sum_{j=1}^{J_\epsilon} \iint_{\omega\times\mathbb{R}^2}\mathbbm{1}_{|v|\leq R_\epsilon}{f}_j^*dydv+O(\epsilon).
\end{aligned}$$
Hence, taking $n\to\infty$, it follows from  that 
$$M=\mathcal{M}(\mathbf{f}^{(n)})\to\sum_{j=1}^{J_\epsilon} \iint_{\omega\times\mathbb{R}^2} \mathbbm{1}_{|v|\leq R_\epsilon}{f}_j^*dydv+O(\epsilon)=\mathcal{M}(\mathbf{f}^*)+O(\epsilon).$$
This proves that $\mathcal{M}(\mathbf{f}^*)=M$, since $\epsilon>0$ is arbitrary.
\end{proof}

Next, we derive the self-consistent equation for the minimizer $\mathbf{f}^*$ obtained in Lemma \ref{reduced problem}.
\begin{lemma}[Self-consistent equation for the partial minimization problem $\mathcal{F}_{\bchi;\textup{min}}(M)$]\label{euler lagrange equation for f}
For $T\ge 0$,  $M>0$ and $\bchi\in\mathcal{A}_{\textup{q.m.}}^\uparrow$, let $\mathbf{f}^*$ be a minimizer for the variational problem $\mathcal{F}_{\bchi;\textup{min}}(M)$. Then, for each $j\geq 1$, $f_j^*$ is a solution to the self-consistent equation 
\begin{equation}\label{lemma: EL equation}
f_j^*(y,v)=\tilde{\beta}\left(\mu -\frac{|v|^2}{2}-h_j(y)\right)
\end{equation}
for some $\mu\in (0,\tfrac{4}{M} (\mathcal{F}_{\bchi;\textup{min}}(M)+T\beta'(1)M )]$, where
$$h_j(y):=\left\langle\left(-\frac{\partial_z^2}{2} +(U_{(\mathbf{f}^*,\bchi)}+V_{\textup{ext}})(y,\cdot)\right)\chi_j(y,\cdot),\chi_j(y,\cdot)\right\rangle_{L^2(0,1)}.
$$
\end{lemma}

\begin{proof}
We prove the lemma following the argument in Dolbeault-Felmer-Lewin \cite[Lemma 3.6]{DFL}. For convenience, we mostly deal with the zero and the positive temperature cases at the same time, because the former case follows simply taking $T=0$ in the proof.

\textbf{\underline{Step 1} (Reduction to the auxiliary problem)}  Let $\bchi\in\mathcal{A}_{\textup{q.m.}}^\uparrow$. For any $\mathbf{f}\in\mathcal{A}_{\textup{c.m.}}$ with $\mathcal{M}(\mathbf{f})=M$ and $\mathcal{I}(\mathbf{f})<\infty$ (see \eqref{auxilary prob} below), its free energy is expanded as 
$$\begin{aligned}
\mathcal{F}_{\bchi}(\mathbf{f})&=\mathcal{F}_{\bchi}(\mathbf{f}^*)+\sum_{j=1}^\infty \iint_{\omega \times\mathbb{R}^2}\left(\frac{|v|^2}{2}+h_j+T\beta'(f_j^*)\right)(f_j-f^*_j)dydv\\
&\quad+T\sum_{j=1}^\infty \iint_{\omega \times\mathbb{R}^2}\left(\beta(f_j)-\beta(f_j^*)-\beta'(f_j^*)(f_j-f_j^*)\right)dydv+\frac{1}{2}\int_{\Omega} |\nabla U_{\rho_{(\mathbf{f}-\mathbf{f}^*,\bchi)}}|^2dx,
\end{aligned}$$
where the series can be rearranged by Lemma \ref{density estimate}, Lemma \ref{improved summability} and  the assumptions on $\beta$ (see Assumption \ref{assump on beta} (ii)).
Since the last two terms on the right hand side are superlinear in $(f_j-f_j^*)$, the second linear term must be non-negative, 
$$\sum_{j=1}^\infty \iint_{\omega \times\mathbb{R}^2}\left(\frac{|v|^2}{2}+h_j+T\beta'(f_j^*)\right)(f_j-f^*_j)dydv\geq0.$$
Consequently, by the convexity of $\beta$, it follows that 
$$\begin{aligned}
&\sum_{j=1}^\infty \iint_{\omega \times\mathbb{R}^2}\left(\frac{|v|^2}{2}+h_j+T\beta'(f_j^*)\right)(f_j-f^*_j)dydv\\
&+T\sum_{j=1}^\infty \iint_{\omega \times\mathbb{R}^2}\left(\beta(f_j)-\beta(f_j^*)-\beta'(f^*_j)(f_j-f_j^*)\right)dydv\geq0
\end{aligned}$$
From this, we deduce that $\mathbf{f}^*$ is a minimizer for the auxiliary variational problem 
$$\mathcal{I}_\textup{min}(M):=\inf\Big\{\mathcal{I}(\mathbf{f}):\mathbf{f}\in\mathcal{A}_{\textup{c.m.}}\textup{ and }\mathcal{M}(\mathbf{f})=M\Big\},$$
where
\begin{equation}\label{auxilary prob}
\mathcal{I}(\mathbf{f}):=\sum_{j=1}^\infty \iint_{\omega \times\mathbb{R}^2}\left(\frac{|v|^2}{2}+h_j\right)f_j+T\beta(f_j)dydv.
\end{equation}
Here, $\mathbf{f}^*$ is a unique minimizer for $\mathcal{I}_\textup{min}(M)$, since $\mathcal{I}(\mathbf{f})$ is convex.

\textbf{\underline{Step 2} (Construction of $\mathbf{f}^*$ from the auxiliary problem)}
To determine $\mu$ in the self-consistent equation \eqref{lemma: EL equation}, we define 
$$F(s)=\mathcal{M}\left(\big\{\tilde{\beta}(s-\tfrac{|v|^2}{2}-h_j(y))\big\}_{j=1}^\infty\right).$$
In the zero temperature case $T=0$, we have
$$F(s)=\mathcal{M}\big(\big\{\mathbbm{1}_{\{\frac{|v|^2}{2}+h_j(y)\le s\}}\big\}_{j=1}^\infty\big)=2\pi\sum_{j=1}^\infty\|(s-h_j(y))_+\|_{L^1(\omega)}.$$
Note that by \eqref{improved summability proof}, $h_j(y)\geq\frac{\pi^2 j^2}{6}$. Hence, for any fixed $s>0$, $\|(s-h_j(y))_+\|_{L^1(\omega)}=0$ except only finitely many $j$'s. Hence, $F(s)$ is finite. Moreover, $F(0)=0$, $F(s)$ is a contunuos increasing  function, and $F(s)\geq 2\pi \|(s-h_1(y))_+\|_{L^1(\omega)}\to\infty$ as $s\to\infty$. Therefore, there exists unique $\mu>0$ such that $F(\mu)=M$. Similarly, in the positive temperature case $T>0$, one can also take $\mu>0$ such that $F(\mu)=M$ from the layer-cake representation 
$$\begin{aligned}
F(s)&=\sum_{j=1}^\infty\int_0^\infty \tilde{\beta}'(t)\left|\left\{(y,v)\in\omega\times\mathbb{R}^2:\frac{|v|^2}{2}+h_j(y)\leq s-t\right\}\right|dt\\
&=2\pi\sum_{j=1}^\infty\int_0^\infty \tilde{\beta}'(t)\|(s-t-h_j(y))_+\|_{L^1(\omega)}dt.
\end{aligned}$$
We claim that $\mathbf{f}^*=\{f_j^*\}_{j=1}^\infty$, defined by 
$$f^*_j(y,v)=\tilde{\beta}\left(\mu-\frac{|v|^2}{2}-h_j(y)\right),$$
is a minimizer for $\mathcal{I}_\textup{min}(M)$. Indeed, near $\mathbf{f}^*$, the functional $\mathcal{I}(\mathbf{f})$ is expanded as 
$$\begin{aligned}
\mathcal{I}(\mathbf{f})&=\mathcal{I}(\mathbf{f}^*)+\sum_{j=1}^\infty \iint_{\omega \times\mathbb{R}^2}\left(\frac{|v|^2}{2}+h_j+T\beta'(f^*_j)\right)(f_j- {f}_j^*)dydv\\
&\quad+T\sum_{j=1}^\infty \iint_{\omega \times\mathbb{R}^2}\Big(\beta(f_j)-\beta(f_j^*)-\beta'(f^*_j)(f_j-f^*_j)\Big)dydv.
\end{aligned}$$
For the second term on the right hand side, using $\mathcal{M}(\mathbf{f}^*)=\mathcal{M}(\mathbf{f})$, we write
$$\begin{aligned}
&\sum_{j=1}^\infty\iint_{\omega \times\mathbb{R}^2}\left(\frac{|v|^2}{2}+h_j+T\beta'(f^*_j)\right)(f_j-f^*_j)dydv\\
&=\sum_{j=1}^\infty\iint_{\omega \times\mathbb{R}^2}\left(\frac{|v|^2}{2}+h_j+T\beta'(f^*_j)-\mu\right)(f_j-f^*_j)dydv.
\end{aligned}$$
We note that by the definition of $\tilde{\beta}$ (see \eqref{beta tilde}), 
$$
\frac{|v|^2}{2}+h_j+T\beta'(f_j^*)=\left\{\begin{aligned}
&\frac{|v|^2}{2}+h_j &&\textup{if }\mu-\frac{|v|^2}{2}-h_j\leq 0,\\
&\mu &&\textup{if }0<\mu-\frac{|v|^2}{2}-h_j<T\beta'(1),\\
&\frac{|v|^2}{2}+h_j+T\beta'(1) &&\textup{if }\mu-\frac{|v|^2}{2}-h_j\geq T\beta'(1).
\end{aligned}\right.$$
Hence, it follows that 
\begin{equation}\label{non-negative second term}
\begin{aligned}
&\sum_{j=1}^\infty\iint_{\omega \times\mathbb{R}^2}\left(\frac{|v|^2}{2}+h_j+T\beta'(f^*_j)-\mu\right)(f_j-f^*_j)dydv\\
&= \sum_{j=1}^\infty\bigg\{ \iint_{\{\mu -\frac{|v|^2}{2}-h_j\leq 0\}}\left(\frac{|v|^2}{2}+h_j-\mu\right)f_jdydv\\
&\quad+\iint_{\{\mu -\frac{|v|^2}{2}-h_j\geq T\beta'(1)\}}\left(\frac{|v|^2}{2}+h_j+T\beta'(1)-\mu\right)(f_j-1)dydv\bigg\}\geq0.
\end{aligned}
\end{equation}
Consequently, by the convexity of $\beta$, we conclude that $\mathcal{I}(\mathbf{f})\geq\mathcal{I}(\mathbf{f}^*)$.

\textbf{\underline{Step 3} (Lagrange multiplier)} It remains to obtain a bound for the Lagrange multiplier $\mu$. 
We observe that for $\mu\geq0$, the free energy has a trivial lower bound, 
$$\begin{aligned}
\mathcal{F}_{\bchi;\textup{min}}(M)&\geq\frac{1}{2}\sum_{j=1}^\infty \iint_{\omega\times\mathbb{R}^2} \left(\frac{|v|^2}{2}+h_j\right)f_j^* dydv\\
&\ge \frac{\mu-T\beta'(1)}{2}\sum_{j=1}^\infty \iint_{\{ 0\le \mu -\frac{|v|^2}{2}-h_j \le T\beta'(1)\} } f_j^* dydv\\
&\quad+\frac{1}{2}\sum_{j=1}^\infty \iint_{\{   \mu -\frac{|v|^2}{2}-h_j > T\beta'(1)\} } \left(\frac{|v|^2}{2}+h_j\right) dydv.
\end{aligned}$$
For the second term in the lower bound, we integrate out the $v$-variable to obtain 
$$\begin{aligned}
& \iint_{\{   \mu -\frac{|v|^2}{2}-h_j > T\beta'(1)\} } \left(\frac{|v|^2}{2}+h_j\right) dvdy\\
&=\pi\int_\omega  \big(\mu-h_j(y)- T\beta'(1)\big)_+^2+2h_j\big(\mu-h_j- T\beta'(1)\big)_+  dy\\
&\geq(\mu- T\beta'(1))\pi \int_\omega \big(\mu-h_j- T\beta'(1)\big)_+  dy.
\end{aligned}$$
Thus, we obtain a lower bound
$$\mathcal{F}_{\bchi;\textup{min}}(M)\geq\frac{\mu-T\beta'(1)}{2}(A+B),$$
where
$$A=\sum_{j=1}^\infty\iint_{\{ 0\le \mu -\frac{|v|^2}{2}-h_j \le T\beta'(1)\} } f_j^* dydv,\quad B=\pi \sum_{j=1}^\infty\int_\omega  \big(\mu-h_j(y)- T\beta'(1)\big)_+  dy.$$
Similarly, the mass can be expressed as 
$$M=\mathcal{M}(\mathbf{f}^*)=A+2B.$$
Thus, it follows that $4\mathcal{F}_{\bchi;\textup{min}}(M)-\mu M\geq -4T\beta'(1)M$. It proves the desired bound on $\mu$.
\end{proof}

\subsection{Proof of Proposition \ref{prop: refinement 2}}
For a minimizing sequence $\{(\mathbf{f}^{(n)},\bchi^{(n)})\}_{n=1}^\infty\subset \mathcal{A}_{\textup{c.m.}}\times\mathcal{A}_{\textup{q.m.}}^\uparrow$ in Proposition \ref{prop: refinement 1}, we refine $\mathbf{f}^{(n)}$ replacing by a minimizer for the variational problem $\mathcal{F}_{\bchi^{(n)};\textup{min}}(M)$ (see \eqref{free energy minimization fixing bchi}). By construction, it is obvious that this refined sequence is also a minimizing sequence for the full variational problem $\mathcal{F}_{\textup{min}}(M)$. Thus, denoting it by $\{(\mathbf{f}^{(n)},\bchi^{(n)})\}_{n=1}^\infty$, we may assume that  
$$f_j^{(n)}=\tilde{\beta}\left(\mu^{(n)} -\frac{|v|^2}{2}-h_j^{(n)}(y)\right),$$
where
$$h_j^{(n)}(y):=\left\langle\left(-\frac{\partial_z^2}{2} +(U_{\rho_{(\mathbf{f}^{(n)},\bchi^{(n)})}}+V_{\textup{ext}})(y,\cdot)\right)\chi_j^{(n)}(y,\cdot), \chi_j^{(n)}(y,\cdot)\right\rangle_{L^2(0,1)}.$$

We claim that by rearrangement, we may impose that $h_j^{(n)}(y)$ is non-decreasing in $j$ (but then $\bchi^{(n)}$ does not need to be included in $\mathcal{A}_{\textup{q.m.}}^\uparrow$). To prove the claim, we fix $n\geq 1$ and $y\in\omega$. Then, there exists a permutation $\sigma^{(n)}(\cdot;y)$ such that $h_{\sigma^{(n)}(j;y)}^{(n)}(y)$ is non-decreasing in $j$, because $h_j^{(n)}(y)\to\infty$ as $j\to\infty$. Now, using $\sigma^{(n)}(j;y)$, we define $(\tilde{\mathbf{f}}^{(n)},\tilde{\bchi}^{(n)})$ by $(\tilde{f}_j^{(n)}(y,v),\tilde{\chi}_j^{(n)}(y,z))=(f_{\sigma^{(n)}(j;y)}^{(n)}(y,v), \chi_{\sigma^{(n)}(j;y)}^{(n)}(y,z))$. Then, since $U_{\rho_{(\tilde{\mathbf{f}}^{(n)},\tilde{\bchi}^{(n)})}}=U_{\rho_{(\mathbf{f}^{(n)},\bchi^{(n)})}}$, we have 
$$\tilde{h}_j^{(n)}(y):=h_{\sigma^{(n)}(j;y)}^{(n)}(y)=\left\langle\left(-\frac{\partial_z^2}{2} +(U_{\rho_{(\tilde{\mathbf{f}}^{(n)},\tilde{\bchi}^{(n)})}}+V_{\textup{ext}})(y,\cdot)\right)\tilde{\chi}_j^{(n)}(y,\cdot), \tilde{\chi}_j^{(n)}(y,\cdot)\right\rangle_{L^2(0,1)}$$
and
\begin{equation}\label{modified minimizing f}
\tilde{f}_j^{(n)}(y,v)=\tilde{\beta}\left(\mu^{(n)} -\frac{|v|^2}{2}-\tilde{h}_j^{(n)}(y)\right).
\end{equation}
On the other hand, $\{(\tilde{\mathbf{f}}^{(n)},\tilde{\bchi}^{(n)})\}_{n=1}^\infty$ is also a minimizing sequence for $\mathcal{F}_{\textup{min}}(M)$ (see Section \ref{sec: rearrangement of admissible pairs}). Hence, replacing $(\mathbf{f}^{(n)},\bchi^{(n)})$ by $(\tilde{\mathbf{f}}^{(n)},\tilde{\bchi}^{(n)})$ but still denoting by $(\mathbf{f}^{(n)},\bchi^{(n)})$, we may assume that $h_j^{(n)}(y)$ is non-decreasing in $j$.

To show $(i)$, we note that by the claim, $\mathbf{f}^{(n)}\in\mathcal{A}_{\textup{c.m.}}^\downarrow$, because non-decreasing $h_j^{(n)}(y)$ in $j$ implies non-increasing $f_j^{(n)}(y,v)=\tilde{\beta} (\mu^{(n)} -\frac{|v|^2}{2}-h_j^{(n)}(y) )$ in $j$. In addition, repeating the argument to prove \eqref{improved summability proof}, one can show that
$$h_j^{(n)}(y)\geq\frac{\pi^2 j^2}{6},$$
while Lemma \ref{euler lagrange equation for f} yields
\begin{equation}\label{uniform mu n bound}
\begin{aligned}
\mu^{(n)}&\leq\frac{4}{M}(\mathcal{F}_{\bchi^{(n)};\textup{min}}(M)+T\beta'(1)M)\\
&\leq \frac{4}{M}(\mathcal{F}(\mathbf{f}^{(n)},\bchi^{(n)})+T\beta'(1)M)\to\frac{4}{M}(\mathcal{F}_{\textup{min}}(M)+T\beta'(1)M) \ \textup{ as }\ n\to \infty.
\end{aligned}
\end{equation}
Therefore, if $j\geq J:=\frac{1}{\pi}\sqrt{\frac{24}{M}(\mathcal{F}_{\textup{min}}(M)+T\beta'(1)M) +1}$, then passing to a subsequence, $\frac{|v|^2}{2}+h_j^{(n)}(y)>\mu^{(n)}$ for all $y\in\omega$, which leads to $f_j^{(n)}\equiv 0$.

For $(ii)$, integrating in $v\in\mathbb{R}^2$, we write
$$\begin{aligned}
\rho_{f_j^{(n)}}(y)&=2\pi\int_{\sqrt{2(\mu^{(n)}-h_j^{(n)}-T\beta'(1))_+}}^{ \sqrt{2(\mu^{(n)}-h_j^{(n)} )_+}} \tilde{\beta}\left(\mu^{(n)} -\frac{r^2}{2}-h_j^{(n)}(y)\right)rdr\\
&\quad+2\pi\left(\mu^{(n)}-h_j^{(n)}(y)- T\beta'(1)\right)_+\\
&\leq 2\pi\mu^{(n)}\Big(\tilde{\beta} (\mu^{(n)})+1\Big).
\end{aligned}$$
Then, using the upper bound on $\mu^{(n)}$ (\eqref{uniform mu n bound}), we conclude that $\rho_{f_j^{(n)}}$ is uniformly bounded.

\section{Third refinement of a minimizing sequence: quantum states as eigenfunctions}\label{sec: 3rd refinement}

By the refinement in the previous sections, we may restrict ourselves to a minimizing sequence $\{(\mathbf{f}^{(n)},\bchi^{(n)})\}_{n=1}^\infty$ in $\mathcal{A}_{\textup{c.m.}}^\downarrow\times\mathcal{A}_{\textup{q.m.}}$ satisfying the properties in Proposition \ref{prop: refinement 2}. The purpose of this section is to refine the quantum state part in $\{(\mathbf{f}^{(n)},\bchi^{(n)})\}_{n=1}^\infty$ considering another ``partial" minimization problem 
\begin{equation}\label{free energy minimization fixing f}
\mathcal{F}_{\mathbf{f};\textup{min}}=\inf\Big\{\mathcal{F}_{\mathbf{f}}(\bchi)=\mathcal{F}(\mathbf{f},\bchi):\ \bchi\in\mathcal{A}_{\textup{q.m.}}\Big\}
\end{equation}
for fixed $\mathbf{f}\in\mathcal{A}_{\textup{c.m.}}^\downarrow$ having only finitely many bands.. We show that a minimizer for the problem $\mathcal{F}_{\mathbf{f};\textup{min}}$ is uniquely determined by a solution to the Schr\"odinger-Poisson equation \eqref{Schrodinger-Poisson} (see Proposition \ref{proposition: reduced problem} below). As a consequence, refining a minimizing sequence even further, we may assume that $\chi_j^{(n)}$'s are eigenfunctions for a certain Schr\"odinger operator, from which compactness will be deduced for quantum states in the next section.

\subsection{Schr\"odinger-Poisson equation}\label{subsection: SP equation}

Given a potential function $U\in L^2(0,1)$, let
$$H[U]=-\frac{\partial_z^2}{2}+U(z)$$
be the Schr\"odinger operator acting on $L^2(0,1)$ with zero boundary condition. Then, the spectrum of the operator $H[U]$ consists of only countably many simple eigenvalues
\begin{equation}\label{eigenvalue: definition}\lambda_1[U]<\lambda_2[U]<\lambda_3[U]<\cdots<\lambda_k[U]<\cdots \to \infty,
\end{equation}
and if we denote an $L^2(0,1)$-normalized $j$-th eigenfunction of $H[U]$ by
\begin{equation}\label{eigenfunction: definition}
\chi_j[U]\in H_0^1(0,1),
\end{equation}
their collection forms an orthonormal basis of $L^2(0,1)$.

In Ben Abdallah-M\'ehats \cite[Proposition 3.5 and 3.6]{BM}, it is shown that for $q>\frac{4}{3}$, if $V_{\textup{ext}}\in L^{q'}(\omega;L^\infty(0,1))$, $\|\{\rho_{f_j}\}_{j=1}^\infty\|_{\ell^1(\mathbb{N};L^q(\omega))}<\infty$ and $\{\rho_{f_j}(y)\}_{j=1}^\infty$ is non-increasing for each $y\in\omega$, then the Schr\"odinger-Poisson equation\footnote{In the equation, the Schr\"odinger equation is given implicitly. Indeed, by definitions, $\chi_j[V+V_{\textup{ext}}]$ represents the solution to  the Schr\"odinger equation $(-\frac{\partial_z^2}{2}+V+V_{\textup{ext}})\chi_j[V+V_{\textup{ext}}]=\lambda_j[V+V_{\textup{ext}}]\chi_j[V +V_{\textup{ext}}]$.}
\begin{equation}\label{Schrodinger-Poisson}
\left\{\begin{aligned}
-\Delta V&=\rho_{(\mathbf{f},\bchi[V+V_{\textup{ext}}])}&&\textup{in }\Omega,\\
V&=0&&\textup{on }\partial\omega\times(0,1),\\
\partial_zV&=0&&\textup{on }\omega\times\{0,1\},
\end{aligned}\right.
\end{equation}
where $\bchi[V+V_{\textup{ext}}]=\big\{\chi_j[V+V_{\textup{ext}}]\big\}_{j=1}^\infty$, has a unique solution $U^*\in H^1(\omega)$.  It holds when $\mathbf{f}\in\mathcal{A}_{\textup{c.m.}}^\downarrow$ and  it has only finitely many bands, provided that $V_{\textup{ext}}$ satisfies the assumption in the main theorem, i.e., $V_{\textup{ext}}\in C(\overline{\Omega})\cap C^1({\Omega})$ and $V_{\textup{ext}}\ge 0$. 

\subsection{Free energy minimization for fixed distributions}

The main result of this section asserts that a free energy minimizer with fixed kinetic distributions can be constructed from the unique solution to the Schr\"odinger-Poisson equation \eqref{Schrodinger-Poisson}.

\begin{proposition}[Free energy minimization for fixed distributions]\label{proposition: reduced problem}
Assume $T\ge 0$ and $\mathbf{f}\in\mathcal{A}_{\textup{c.m.}}^\downarrow$ has only finitely many bands. Let   $U^*$ be the unique solution to the Schrodinger-Poisson equation \eqref{Schrodinger-Poisson} with $\mathcal{F}(\mathbf{f},\bchi^*)<\infty$,   where $\bchi^*=\bchi[U^*+V_{\textup{ext}}]$. Then $\bchi^*$ is a  unique minimizer for the variational problem $\mathcal{F}_{\mathbf{f},\min}$ in the sense that if $\tilde{\bchi}^*$ is a minimizer for $\mathcal{F}_{\mathbf{f},\min}$, then $U_{\rho_{(\mathbf{f},\bchi^*)}}=U_{\rho_{(\mathbf{f},\tilde{\bchi}^*)}}$.
\end{proposition}

\begin{remark}
The minimization problem $\mathcal{F}_{\mathbf{f},\min}$ is not easy to solve by the concentration-compactness principle (see Remark \ref{difficulties} $(ii)$). Instead, using the coercivity of the free energy, we directly show that a minimizer is obtained from $\bchi^*=\bchi[U^*+V_{\textup{ext}}]$. Nevertheless, our proof implicitly uses the concentration-compactness principle in that the solution $U^*$ to the Schr\"odinger-Poisson equation \eqref{Schrodinger-Poisson} is constructed as a minimizer of an auxiliary variational problem \cite{BM}.
\end{remark}

The following coercivity estimate is crucial in proving the main result of this section, as well as uniqueness and stability of a free energy minimizer (see Section \ref{exist and uniq stable}).

\begin{lemma}[Coercivity of the free energy]\label{lemma: energy coercivity}
Assume $T\ge 0$, $\mathbf{f}\in\mathcal{A}_{\textup{c.m.}}^\downarrow$ has  only finitely many non-zero subbands  and $(\mathbf{f},\bchi^*)\in  \mathcal{A}_{\textup{c.m.}}^\downarrow \times\mathcal{A}_{\textup{q.m.}}$ with 
$$\bchi^*=\bchi[U^*+V_{\textup{ext}}] \ \textup{ and } \ \mathcal{F}(\mathbf{f},\bchi^*)<\infty,$$
where $U^*$ is the solution to the Schrodinger-Poisson equation \eqref{Schrodinger-Poisson}. Then, for any $(\tilde{\mathbf{f}},\tilde{\bchi})\in\mathcal{A}_{\textup{c.m.}} \times\mathcal{A}_{\textup{q.m.}}$, we have 
\begin{align*}
\mathcal{F}(\tilde{\mathbf{f}},\tilde{\bchi})-\mathcal{F}(\mathbf{f},\bchi^*)&\geq\frac{1}{2}\|\nabla (U_{\rho_{(\tilde{\mathbf{f}},\tilde{\bchi})}}-U_{\rho_{(\mathbf{f},\bchi^*)}})\|_{L^2(\Omega)}^2 \\
&\quad +\sum_{j=1}^\infty \iint_{\omega\times\mathbb{R}^2}\left(\frac{|v|^2}{2}+\lambda_j^*+T\beta'\big(f_j \big)\right)(\tilde{f}_j^{\sigma_\downarrow}-f_j) dy dv,
\end{align*}
where $\lambda_j^*=\lambda_j^*[U^*+V_{\textup{ext}}]$ and $\tilde{\mathbf{f}}^{\sigma_\downarrow}\in \mathcal{A}_{\textup{c.m.}}^\downarrow $ (see Lemma \ref{rearrange distrib} for the definition of $\sigma_\downarrow$). 
\end{lemma}
\begin{proof}
By Lemma \ref{rearrange distrib}, we rearrange $\tilde{\mathbf{f}}\in \mathcal{A}_{\textup{c.m.}}$ by non-increasing $\tilde{\mathbf{f}}^{\sigma_\downarrow}\in \mathcal{A}_{\textup{c.m.}}^\downarrow$. Note that the total density function and the mass and the free energy are invariant; $\rho_{(\tilde{\mathbf{f}}^{\sigma_\downarrow},\tilde{\bchi}^{\sigma_\downarrow})}=\rho_{(\tilde{\mathbf{f}},\tilde{\bchi})}$, $\mathcal{M}(\tilde{\mathbf{f}}^{\sigma_\downarrow})=\mathcal{M}(\tilde{\mathbf{f}})$ and $\mathcal{F}(\tilde{\mathbf{f}}^{\sigma_\downarrow},\tilde{\bchi}^{\sigma_\downarrow})=\mathcal{F}(\tilde{\mathbf{f}},\tilde{\bchi})$. For convenience, we denote $\tilde{U}=U_{\rho_{(\tilde{\mathbf{f}},\tilde{\bchi})}}$, $H^*=H[U^*+V_\textup{ext}]=-\frac{\partial_z^2}{2}+U^*+V_\textup{ext}$ and $\tilde{H}=H[\tilde{U}+V_\textup{ext}]=-\frac{\partial_z^2}{2}+\tilde{U}+V_\textup{ext}$. Then, the difference of the two free energies with $\mathcal{F}(\tilde{\mathbf{f}}^{\sigma_\downarrow},\tilde{\bchi}^{\sigma_\downarrow})=\mathcal{F}(\tilde{\mathbf{f}},\tilde{\bchi})$ can be written as
$$\begin{aligned}
\mathcal{F}(\tilde{\mathbf{f}},\tilde{\bchi})-\mathcal{F}(\mathbf{f},\bchi^*)&=\sum_{j=1}^\infty \iint_{\omega\times\mathbb{R}^2}\frac{|v|^2}{2}(\tilde{f}_j^{\sigma_\downarrow}-f_j) dy dv\\
&\quad+\sum_{j=1}^\infty \iint_{\omega\times\mathbb{R}^2}\left(\langle \tilde{H}\tilde{\chi}_j^{\sigma_\downarrow}, \tilde{\chi}_j^{\sigma_\downarrow}\rangle_{L^1(0,1)}\tilde{f}_j^{\sigma_\downarrow}-\lambda_j^*f_j\right) dy dv\\
&\quad-\frac{1}{2}\|\nabla \tilde{U}\|_{L^2(\Omega)}^2+\frac{1}{2}\|\nabla U^*\|_{L^2(\Omega)}^2\\
&\quad+\sum_{j=1}^\infty\iint_{\omega \times\mathbb{R}^2}\left(\beta(\tilde{f}_j^{\sigma_\downarrow})-\beta(f_j)\right)dydv.
\end{aligned}$$
For the second term on the right hand side, we expand 
$$\begin{aligned}
&\langle \tilde{H}\tilde{\chi}_j^{\sigma_\downarrow}, \tilde{\chi}_j^{\sigma_\downarrow}\rangle_{L^2(0,1)}\tilde{f}_j^{\sigma_\downarrow}-\lambda_j^* f_j\\
&=\Big(\langle \tilde{H}\tilde{\chi}_j^{\sigma_\downarrow}, \tilde{\chi}_j^{\sigma_\downarrow}\rangle_{L^2(0,1)}-\langle H^*\tilde{\chi}_j^{\sigma_\downarrow}, \tilde{\chi}_j^{\sigma_\downarrow}\rangle_{L^2(0,1)}\Big)\tilde{f}_j^{\sigma_\downarrow}\\
&\quad+\Big(\langle H^*\tilde{\chi}_j^{\sigma_\downarrow}, \tilde{\chi}_j^{\sigma_\downarrow}\rangle_{L^2(0,1)}-\lambda_j^*\Big)\tilde{f}_j^{\sigma_\downarrow}+\lambda_j^*(\tilde{f}_j^{\sigma_\downarrow}-f_j)\\
&=\int_0^1 (\tilde{U}-U^*)\left(\tilde{f}_j^{\sigma_\downarrow}|\tilde{\chi}_j^{\sigma_\downarrow}|^2\right)dz+\Big(\langle H^*\tilde{\chi}_j^{\sigma_\downarrow}, \tilde{\chi}_j^{\sigma_\downarrow}\rangle_{L^2(0,1)}-\lambda_j^*\Big)\tilde{f}_j^{\sigma_\downarrow}+\lambda_j^*(\tilde{f}_j^{\sigma_\downarrow}-f_j).
\end{aligned}$$
Then, it follows that 
$$\begin{aligned}
\mathcal{F}(\tilde{\mathbf{f}},\tilde{\bchi})-\mathcal{F}(\mathbf{f},\bchi^*)&=\sum_{j=1}^\infty \iint_{\omega\times\mathbb{R}^2}\left(\frac{|v|^2}{2}+\lambda_j^*\right)(\tilde{f}_j^{\sigma_\downarrow}-f_j) dy dv\\
&\quad+\sum_{j=1}^\infty \iint_{\omega\times\mathbb{R}^2}\left(\langle \tilde{H}\tilde{\chi}_j^{\sigma_\downarrow}, \tilde{\chi}_j^{\sigma_\downarrow}\rangle_{L^1(0,1)}-\lambda_j^*\right)\tilde{f}_j^{\sigma_\downarrow} dy dv\\
&\quad+\frac{1}{2}\|\nabla (\tilde{U}-U^*)\|_{L^2(\Omega)}^2+\sum_{j=1}^\infty\iint_{\omega \times\mathbb{R}^2}\left(\beta(\tilde{f}_j^{\sigma_\downarrow})-\beta(f_j)\right)dydv,
\end{aligned}$$
where we used $\sum_{j=1}^\infty\rho_{\tilde{f}_j^{\sigma_\downarrow}}|\tilde{\chi}_j^{\sigma_\downarrow}|^2=\sum_{j=1}^\infty\rho_{\tilde{f}_j}|\tilde{\chi}_j|^2=\rho_{(\tilde{\mathbf{f}},\tilde{\bchi})}$ and 
$$\int_\Omega(\tilde{U}-U^*)\rho_{(\tilde{\mathbf{f}},\tilde{\bchi})}dx=\big\langle\nabla(\tilde{U}-U^*), \nabla\tilde{U}\big\rangle_{L^2(\Omega)}.$$
For the second term on the right hand side, we observe that for any $J\geq 1$ and $y\in\omega$,
$$\begin{aligned}
\sum_{j=1}^{J}\Big(\langle H^*\tilde{\chi}_j^{\sigma_\downarrow}, \tilde{\chi}_j^{\sigma_\downarrow}\rangle_{L^1(0,1)}-\lambda_j^*\Big)\tilde{f}_j^{\sigma_\downarrow}&=\tilde{f}_{J}^{\sigma_\downarrow}\sum_{j=1}^{J}\Big(\langle H^*\tilde{\chi}_j^{\sigma_\downarrow}, \tilde{\chi}_j^{\sigma_\downarrow}\rangle_{L^1(0,1)}-\lambda_j^*\Big)\\
&\quad+\sum_{k=1}^{J-1}(\tilde{f}_k^{\sigma_\downarrow}-\tilde{f}_{k+1}^{\sigma_\downarrow})\sum_{j=1}^k\Big(\langle H^*\tilde{\chi}_j^{\sigma_\downarrow}, \tilde{\chi}_j^{\sigma_\downarrow}\rangle_{L^1(0,1)}-\lambda_j^*\Big)\\
&\geq0,
\end{aligned}$$
since $\tilde{f}_j^{\sigma_\downarrow}$ is non-increasing and the min-max principle implies $\sum_{j=1}^k(\langle H^*\tilde{\chi}_j^{\sigma_\downarrow}, \tilde{\chi}_j^{\sigma_\downarrow}\rangle_{L^2(0,1)}-\lambda_j^*)\ge 0$ for all $k=1,2,\cdots, J$. Hence, by density argument, we obtain that 
$$\sum_{j=1}^\infty\iint_{\omega\times\mathbb{R}^2}\left(\langle \tilde{H}\tilde{\chi}_j^{\sigma_\downarrow}, \tilde{\chi}_j^{\sigma_\downarrow}\rangle_{L^2(0,1)}-\lambda_j^*\right)\tilde{f}_j^{\sigma_\downarrow}dydv\geq0.$$
Consequently, it follows that
$$\begin{aligned}
\mathcal{F}(\tilde{\mathbf{f}},\tilde{\bchi})-\mathcal{F}(\mathbf{f},\bchi^*)&\ge\frac{1}{2}\|\nabla (\tilde{U}-U^*)\|_{L^2(\Omega)}^2\\
&\quad+\sum_{j=1}^\infty \iint_{\omega\times\mathbb{R}^2}\left(\frac{|v|^2}{2}+\lambda_j^*+\beta'\big(f_j \big)\right)(\tilde{f}_j^{\sigma_\downarrow}-f_j) dy dv\\
&\quad+\sum_{j=1}^\infty\iint_{\omega \times\mathbb{R}^2}\left(\beta(\tilde{f}_j^{\sigma_\downarrow})-\beta(f_j)-\beta'(f_j)(\tilde{f}_j^{\sigma_\downarrow}-f_j)\right)dydv.
\end{aligned}$$
Therefore, the lemma follows from the convexity of $\beta$.
\end{proof} 

\begin{proof}[Proof of Proposition \ref{proposition: reduced problem}]
Fix $\mathbf{f}\in\mathcal{A}_{\textup{c.m.}}^\downarrow$ having finitely many bands and let $\bchi^*=\bchi[U^*+V_{\textup{ext}}]$. By Lemma \ref{lemma: energy coercivity}, for $\bchi\in \mathcal{A}_{\textup{q.m.}}$, we have
$$\mathcal{F}_{\mathbf{f}}(\bchi)-\mathcal{F}_{\mathbf{f}}(\bchi^*)\ge \frac{1}{2}\|\nabla (U_{\rho_{(\mathbf{f},\bchi)}}-U_{\rho_{(\mathbf{f},\bchi^*)}})\|_{L^2(\Omega)}^2.$$
Thus, it follows that $\bchi^*$ is a  minimizer for the variational problem $\mathcal{F}_{\mathbf{f},\min}$. Moreover, if  $\tilde{\bchi}^*$ is a minimizer for $\mathcal{F}_{\mathbf{f},\min}$, then $U_{\rho_{(\mathbf{f},\bchi^*)}}=U_{\rho_{(\mathbf{f},\tilde{\bchi}^*)}}$ because of the Dirichlet boundary condition on $\partial\omega\times (0,1)$ (see \eqref{Poisson equation}).
 \end{proof}

\section{Construction of a free energy minimizer, and its uniqueness and stability: Proof of the main results (Theorem \ref{existence of a mini} and \ref{stability})}\label{exist and uniq stable}

\subsection{Proof of Theorem \ref{existence of a mini}}
Fix $T\geq 0$ and $M>0$, and let $\{(\mathbf{f}^{(n)},\bchi^{(n)})\}_{n=1}^\infty$ be a minimizing sequence for the full variational problem $\mathcal{F}_{\textup{min}}(M)=\mathcal{F}_{\textup{min}}(M; T,\beta)$ in Proposition \ref{prop: refinement 2}. For each $n$, we replace $\bchi^{(n)}$  by the minimizer of the partial problem $\mathcal{F}_{\mathbf{f}^{(n)},\min}$ (see Proposition \ref{proposition: reduced problem}). In this way, we construct a minimizing sequence $\{(\mathbf{f}^{(n)},\bchi^{(n)})\}_{n=1}^\infty\subset \mathcal{A}_{\textup{c.m.}}^\downarrow\times\mathcal{A}_{\textup{q.m.}}$ such that 
\begin{enumerate}[$(i)$]
\item There exists $J\geq1$, independent of $n\geq 1$, such that $f_j^{(n)}\equiv0$ for all $j\geq J+1$;
\item $\rho_{f_j^{(n)}}(y)$ are bounded uniformly in $y\in\omega$ and $j,n\geq1$;
\item $\bchi^{(n)}=\bchi[U^{(n)}]$ (see the definition \eqref{eigenfunction: definition}), where $U^{(n)}:=U_{\rho_{(\mathbf{f}^{(n)},\bchi^{(n)})}}+V_{\textup{ext}}$.
\end{enumerate}
In a sequel, we show that the refined quantum states $\{\bchi^{(n)}\}_{n=1}^\infty$ are equicontinuous on $\Omega$. Then, taking the (weak) limit of the minimizing sequence, we prove that the limit is indeed a minimizer for the full variational problem. 

\subsubsection{Equicontinuity of quantum states}
We claim that for each $\chi_j^{(n)}$ satisfies 
\begin{equation}\label{eq: equicontinuity}
\sup_{n\ge1}\|\chi_j^{(n)}\|_{L^\infty (\omega; H_0^1 (0,1))\cap C^{0,\frac{1}{2}}(\Omega)} \lesssim j ,
\end{equation}
where the implicit constant is independent of $n\geq 1$ and 
$$\|u\|_{C^{0,\frac{1}{2}}(\Omega)}=\|u\|_{C(\Omega)}+\sup_{x,x'\in\Omega,\ x\neq x'}\frac{|u(x)-u(x')|}{|x-x'|^\frac12}.$$

The proof of \eqref{eq: equicontinuity} relies on the well-known fact for the 1D Schr\"odinger operator discussed in Section \ref{subsection: SP equation}; the eigenvalues and the corresponding  eigenfunctions are stable under potential perturbations (see \cite[Lemma 2.4]{BM} and \cite[Chapter 2]{PT}).
\begin{lemma}[Stability for 1D Schr\"odinger operator]\label{eigen estimates}
If $U, V\in L^2(0,1)$, then there exists $C>0$, independent of $U$, $V$ and $j$, such that
$$|\lambda_j[U]-\lambda_j[V]|+\|\chi_j[U]-\chi_j[V]\|_{L^\infty(0,1)}\le Ce^{C(\|U\|_{L^2(0,1)}+\|V\|_{L^2(0,1)})}\|U-V\|_{L^1(0,1)}.$$
\end{lemma}

We recall from Lemma \ref{basic estimates} that 
$$\sup_{n\geq1}\|\rho_{(\mathbf{f}^{(n)}, \bchi^{(n)})}\|_{L^\frac53 (\Omega)}<\infty.$$
On the other hand, $U_{\rho_{(\mathbf{f}^{(n)},\bchi^{(n)})}}$ satisfies the Neumann boundary condition on $\omega\times \{0,1\}$, so by the even reflection, we may assume that 
$U_{\rho_{(\mathbf{f}^{(n)},\bchi^{(n)})}}\in H^1(\omega\times(-\frac12,\frac32))$ satisfies 
$$\left\{\begin{aligned} 
-\Delta U_{\rho_{(\mathbf{f}^{(n)},\bchi^{(n)})}}&=\rho_{(\mathbf{f}^{(n)},\bchi^{(n)})}&&\textup{ in } \omega\times(-\tfrac12,\tfrac32), \\
 U_{\rho_{(\mathbf{f}^{(n)},\bchi^{(n)})}}&=0 &&\textup{ on } \partial \omega \times (-\tfrac12,\tfrac32).
\end{aligned}\right.$$
Thus, applying elliptic estimates (see \cite[Theorem 9.13]{GT}) and the Sobolev embedding, we obtain
$$\sup_{n\geq 1}\|U_{\rho_{(\mathbf{f}^{(n)},\bchi^{(n)})}} \|_{L^\infty (\Omega)}<\infty.$$ 

By Lemma \ref{eigen estimates} with $U^{(n)}=U_{\rho_{(\mathbf{f}^{(n)},\bchi^{(n)})}}+V_{\textup{ext}}$ and $\lambda_j[0]=(\pi j)^2$ (see \eqref{eigenvalue: definition} and \eqref{eigenfunction: definition} for the notations), we have
$$\big|\lambda_j[U^{(n)}(y,\cdot)]-(\pi j)^2\big| \le Ce^{C\|U^{(n)}\|_{L^2(0,1)}} \|U^{(n)}\|_{L^1(0,1)}\lesssim 1.$$
Then, by the Sobolev embedding and elliptic estimates (see \cite[Theorem 9.13]{GT}), we obtain that 
$$\|\chi_j^{(n)}(y,\cdot)\|_{L^\infty(0,1)}^2\lesssim\|\chi_j^{(n)}(y,\cdot)\|_{H_0^1 (0,1)}^2\leq 1+\lambda_j[U^{(n)}(y,\cdot)] \lesssim j^2$$
for all $y\in\omega$. Hence, it follows that 
$$\|U_{\rho_{(\mathbf{f}^{(n)},\bchi^{(n)})}} \|_{C^1(\overline{\Omega})}\lesssim \|\rho_{(\mathbf{f}^{(n)},\bchi^{(n)})} \|_{L^\infty(\overline{\Omega})}\leq\sum_{j=1}^J\|\rho_{f_j^{(n)}}\|_{L^\infty(\omega)}\|\chi_j^{(n)}\|_{L^\infty(\Omega)}^2 \le J^3.$$
since $f_j^{(n)}\equiv0$ for all $j\geq J+1$ and $\rho_{f_j^{(n)}}(y)$'s are bounded uniformly in $y\in\omega$. Consequently, by Lemma \ref{eigen estimates} again, we prove that 
$$|{\chi}_j^{(n)}(y,z)- {\chi}_j^{(n)}(y',z)|\lesssim \|U^{(n)}(y,\cdot)-U^{(n)}(y',\cdot)\|_{L^1(0,1)}\lesssim_J|y-y'|,$$
where we used the regularity assumption  $V_\textup{ext}\in C^1(\overline{\Omega})$, while by the fundamental theorem of calculus, 
$$\begin{aligned}
|{\chi}_j^{(n)}(y,z)- {\chi}_j^{(n)}(y,z')|&=\left|\int_{z'}^z  (\partial_s\chi_j^{(n)}) (y,s)ds\right|\le \|\partial_z \chi_j^{(n)}(y,\cdot)\|_{L^2(0,1)}|z-z'|^\frac12\\
&\lesssim j|z-z'|^\frac12.
\end{aligned}$$
Combining the two inequalities, we conclude that $|{\chi}_j^{(n)}(x)- {\chi}_j^{(n)}(x')|\lesssim j|x-x'|^\frac12$.

\subsubsection{Existence of a minimizer}
For the kinetic distribution part $\{\mathbf{f}^{(n)}\}_{n=1}^\infty\subset \mathcal{A}_{\textup{c.m.}}^\downarrow$, taking the weak subsequential limit as in the proof of Lemma \ref{reduced problem}, we have $f_j^{(n)}\rightharpoonup f_j^*$ in $L^r(\omega\times \R^2)$ for any $1<r<\infty$ so that $\mathbf{f}^*\in \mathcal{A}_{\textup{c.m.}}^\downarrow$.
 For the quantum part $\{\bchi^{(n)}\}_{n=1}^\infty$, by the equicontinuity \eqref{eq: equicontinuity} and the Arzel\`{a}-Ascoli theorem, we deduce that up to a subsequence, $\|\chi_j^{(n)}-\chi_j^*\|_{C(\Omega)}\to0$ as well $\chi_j^{(n)}(y,\cdot)\rightharpoonup\chi_j^*(y,\cdot)$ in $H^1(0,1)$ as $n\to\infty$ as for each $y\in\omega$.

We claim that $(\mathbf{f}^*,\bchi^*)$ is a minimizer for the problem $\mathcal{F}_\textup{min}(M)$. Indeed, we have 
$$\delta_{jk}=\langle\chi_j^{(n)}, \chi_k^{(n)}\rangle_{L^2(0,1)}\to\langle\chi_j^*, \chi_k^*\rangle_{L^2(0,1)}.$$
Instead of \eqref{quantum energy lower bound}, by a diagonal argument, we obtain 
$$\begin{aligned}
&\liminf_{n\to \infty}\sum_{j=1}^\infty\iint_{\omega\times\mathbb{R}^2} \langle (-\tfrac{\partial_z^2}{2}+V_{\textup{ext}}) \chi_j^{(n)}, \chi_j^{(n)}\rangle_{L^2 (0,1)}{f}_j^{(n)}dydv\\
&\ge \sum_{j=1}^\infty\iint_{\omega\times\mathbb{R}^2}\langle (-\tfrac{\partial_z^2}{2}+V_{\textup{ext}}) \chi_j^*, \chi_j^*\rangle_{L^2 (0,1)}{f}_j^* dydv.
\end{aligned}$$
In addition, we have 
$$\begin{aligned}
\rho_{(\mathbf{f}^{(n)}, \bchi^{(n)})}-\rho_{(\mathbf{f}^*, \bchi^*)}&=\sum_{j=1}^J\rho_{f_j^{(n)}}\big(|\chi_j^{(n)}|^2-|\chi_j^*|^2\big)+\big(\rho_{(\mathbf{f}^{(n)}, \bchi^*)}-\rho_{(\mathbf{f}^*, \bchi^*)}\big)\rightharpoonup 0\textup{ in }L^{\frac{6}{5}}(\Omega),
\end{aligned}$$
because $\|\chi_j^{(n)}-\chi_j^*\|_{C(\Omega)}\to0$ and $\rho_{(\mathbf{f}^{(n)}, \bchi^*)}\rightharpoonup\rho_{(\mathbf{f}^*, \bchi^*)}$ in $L^\frac{6}{5}(\Omega)$ (by  \eqref{l3weak}). Consequently, it follows that 
$$\liminf_{n\to \infty}\int_\Omega |\nabla  U_{\rho_{(\mathbf{f}^{(n)},\bchi^{(n)})}}  |^2dx\ge \int_\Omega |\nabla  U_{\rho_{(\mathbf{f}^*,\bchi^*)}}|^2dx.$$
Hence, repeating the proof of Lemma \ref{reduced problem}, one can show that  $(\mathbf{f}^*,\bchi^*)$ is admissible, $M= \mathcal{M}(\mathbf{f}^{(n)})\to\mathcal{M}(\mathbf{f}^*)$ and $\mathcal{F}_{ \textup{min}}(M)=\underset{n\to \infty}{\liminf} \mathcal{F}(\mathbf{f}^{(n)}, \bchi^{(n)})\ge  \mathcal{F}(\mathbf{f}^*, \bchi^*)$. Therefore, $(\mathbf{f}^*,\bchi^*)\in \mathcal{A}_{\textup{c.m.}}^\downarrow\times\mathcal{A}_{\textup{q.m.}}$ is a free energy minimizer.

\subsubsection{Self-consistent equation \eqref{self consi} and \eqref{Schro Poisson} and uniqueness}
Let $(\mathbf{f}^*,\bchi^*)\in\mathcal{A}_{\textup{c.m.}}^\downarrow\times\mathcal{A}_{\textup{q.m.}}$ be a free energy minimizer for $\mathcal{F}_\textup{min}(M)$. To derive the eigenvalue equation \eqref{Schro Poisson}, we note that $\bchi^*$ is the unique minimizer for the partial variational problem for fixed $\mathbf{f}^*$, i.e., $\mathcal{F}(\mathbf{f}^*,\bchi^*)=\mathcal{F}_{\mathbf{f}^*;\textup{min}}$. Hence, by Proposition \ref{proposition: reduced problem}, $\bchi^*$ satisfies the Schr\"odinger equation \eqref{Schro Poisson}. On the other hand, $\mathbf{f}^*$ is a minimizer for another partial variational problem for fixed $\bchi^*$, i.e., $\mathcal{F}(\mathbf{f}^*,\bchi^*) =\mathcal{F}_{\bchi^*;\textup{min}}(M)$. Thus, by Lemma \ref{euler lagrange equation for f}, we deduce the self-consistent equation \eqref{self consi}.

For uniqueness, we assume that there is another minimizer $(\tilde{\mathbf{f}}^*,\tilde{\bchi}^*)$ with $U_{\rho_{(\mathbf{f}^*,\bchi^*)}}\not\equiv  U_{\rho_{(\tilde{\mathbf{f}}^*,\tilde{\bchi}^*)}}$. Here, we may assume that $(\tilde{\mathbf{f}}^*,\tilde{\bchi}^*)\in \mathcal{A}_{\textup{c.m.}}^\downarrow\times\mathcal{A}_{\textup{q.m.}}$ because the free energy and the total density $\rho_{(\mathbf{f} , \bchi )}=\sum_{j=1}^\infty \rho_{f_j}(y)|\chi_j(x)|^2$ are invariant under the rearrangement in $j$-summed quantities.
Then, by Lemma \ref{lemma: energy coercivity} and the estimate in \eqref{non-negative second term}, we have
$$\begin{aligned}
0=\mathcal{F}(\tilde{\mathbf{f}}^*,\tilde{\bchi}^*)-\mathcal{F}(\mathbf{f}^*,\bchi^*)&\geq\frac{1}{2}\|\nabla (U_{\rho_{(\tilde{\mathbf{f}}^*,\tilde{\bchi}^*)}}-U_{\rho_{(\mathbf{f}^*,\bchi^*)}})\|_{L^2(\Omega)}^2\\
&\quad +\sum_{j=1}^\infty \iint_{\omega\times\mathbb{R}^2}\left(\frac{|v|^2}{2}+\lambda_j^*+T\beta'\big(f_j^* \big)\right)(\tilde{f}^*_j-f^*_j) dy dv\\
&\geq\frac{1}{2}\|\nabla (U_{\rho_{(\tilde{\mathbf{f}}^*,\tilde{\bchi}^*)}}-U_{\rho_{(\mathbf{f}^*,\bchi^*)}})\|_{L^2(\Omega)}^2,
\end{aligned}$$
which deduces a contradiction.

\subsubsection{Structure}
It is known from the spectral theory \cite{BM, BM Diffusiv, PT}  that the linear operator $-\frac{\partial_z^2}{2}+ (U_{\rho_{(\mathbf{f}^*,\bchi^*)}}+V_{\textup{ext}})(y,\cdot)$ has only simple eigenvalues so that $\lambda_j^*(y)$ is strictly increasing in $j\geq 1$. Moreover, $\lambda_j^*(y)$ has a lower bound $\geq\frac{(\pi j)^2}{3}$ (following the proof of \eqref{improved summability proof}). Therefore, for every $(y,v)\in\omega\times\mathbb{R}^2$, $f_j(y,v)$ is strictly decreasing in $j$ and $f_j^*(y,v)\equiv0$ for $j\geq\frac{\sqrt{3\mu}}{\pi}$.

\subsection{Proof of Theorem \ref{stability}}

Note that $(\mathbf{f}^*,\bchi^*)$ is a weak solution to the time-dependent Vlasov-Schr\"odinger-Poisson system, because $\mathbf{f}^*$ is a function of the microscopic energy $\frac{|v|^2}{2}-\lambda_j(y)$ and $\bchi^* =\bchi [U_{(\mathbf{f}^*,\bchi^*)}+V_{\textup{ext}}]$. For stability, we observe from the assumptions on the initial data and the conservation laws,  and the invariant property under the rearrangement that 
$$\begin{aligned}
\left(1+\mu\right)\delta&\geq\mathcal{F}(\mathbf{f}_0,\bchi_0)-\mathcal{F}(\mathbf{f}^*,\bchi^*)-\mu\big(\mathcal{M}(\mathbf{f}_0)-M\big)\\
&=\mathcal{F}(\mathbf{f}(t),\bchi(t))-\mathcal{F}(\mathbf{f}^*,\bchi^*)-\mu\big(\mathcal{M}(\mathbf{f}(t))-\mathcal{M}(\mathbf{f}^*)\big)\\
&=\mathcal{F}(\mathbf{f}^{\sigma_\downarrow}(t),\bchi^{\sigma_\downarrow}(t))-\mathcal{F}(\mathbf{f}^*,\bchi^*)-\mu\big(\mathcal{M}(\mathbf{f}^{\sigma_\downarrow}(t))-\mathcal{M}(\mathbf{f}^*)\big),
\end{aligned}$$
where $(\mathbf{f}^{\sigma_\downarrow}(t),\bchi^{\sigma_\downarrow}(t))\in\mathcal{A}_{\textup{c.m.}}^\downarrow\times\mathcal{A}_{\textup{q.m.}}$ (see Lemma \ref{rearrange distrib} for the definition of $\sigma_\downarrow$).
Hence, it follows from Lemma \ref{lemma: energy coercivity} and the estimate in \eqref{non-negative second term} that 
\begin{align*}
\left(1+\mu\right)\delta&\geq\frac{1}{2}\|\nabla (U_{\rho_{(\mathbf{f}(t),\bchi(t))}}-U_{\rho_{(\mathbf{f}^*,\bchi^*)}})\|_{L^2(\Omega)}^2 \\
&\quad +\sum_{j=1}^\infty \iint_{\omega\times\mathbb{R}^2}\left(\frac{|v|^2}{2}+\lambda_j^*+T\beta'\big(f_j^* \big)-\mu\right)(f_j^{\sigma_\downarrow}(t)-f_j^*) dy dv\\
&\geq\frac{1}{2}\|\nabla (U_{(\mathbf{f}(t),\bchi(t))}-U_{(\mathbf{f}^*,\bchi^*)})\|_{L^2(\Omega)}^2.
\end{align*}
Therefore, taking $\delta=\frac{\epsilon}{2(1+\mu)}$, we prove the theorem.

\end{document}